\documentclass[a4paper,leqno,12pt]{amsart}
\usepackage{amsfonts}
\usepackage{graphicx}
\usepackage{amsmath}

\newtheorem{theorem}{Theorem}

\newtheorem{corollary}[theorem]{Corollary}

\newtheorem{definition}{Definition}

\newtheorem{lemma}[theorem]{Lemma}

\newtheorem{proposition}[theorem]{Proposition}
\newtheorem{remark}{Remark}

\begin{document}
\title[Generalized Gevrey ultradistributions]{Generalized Gevrey
ultradistributions}
\author{Khaled BENMERIEM}
\author{Chikh BOUZAR}
\address{Centre Universitaire de Mascara. Algeria }
\email{benmeriemkhaled@yahoo.com }
\address{Department of Mathematics, Oran-Essenia University. Algeria }
\email{bouzar@yahoo.com ; bouzar@univ-oran.dz}
\date{}
\maketitle

\begin{abstract}
We first introduce new algebras of generalized functions
containing Gevrey ultradistributions and then develop a Gevrey
microlocal analysis suitable for these algebras. Finally, we give
an application through an extension of the well-known
H\"{o}rmander's theorem on the wave front of the product of two
distributions.
\end{abstract}

\section{Introduction}

The theory of generalized functions initiated by J. F. Colombeau, see \cite%
{Col2} and \cite{Col3}, in connection with the problem of multiplication of
Schwartz distributions \cite{Schw}, has been developed and applied in
nonlinear and linear problems, \cite{Col3}, \cite{Ober} and \cite{NPS}. The
recent book \cite{GKOS} gives further developments and applications of such
generalized functions. Some methods of constructing algebras of generalized
functions of Colombeau type are given in \cite{A-R}, \cite{GKOS} and \cite%
{Mart}.

Ultradistributions, important in theoretical as well applied fields, see
\cite{Kom}, \cite{LM} and \cite{Rod}, are natural generalization of Schwartz
distributions, and the problem of multiplication of ultradistributions is
still posed. So, it is natural to search for algebras of generalized
functions containing spaces of ultradistributions, to study and to apply
them. This is the purpose of this paper.

First, we introduce new differential algebras of generalized Gevrey
ultradistributions $\mathcal{G}^{\sigma }\left( \Omega \right) $ defined on
an open set $\Omega $\ of $\mathbb{R}^{n}$ as the quotient algebra%
\begin{equation*}
\mathcal{G}^{\sigma }\left( \Omega \right) =\frac{\mathcal{E}_{m}^{\sigma
}\left( \Omega \right) }{\mathcal{N}^{\sigma }\left( \Omega \right) }\text{
\ ,}
\end{equation*}%
where $\mathcal{E}_{m}^{\sigma }(\Omega )$ is the space of $\left(
f_{\varepsilon }\right) _{\varepsilon }\in C^{\infty }\left( \Omega \right)
^{\left] 0,1\right] }$ satisfying for every compact $K$ of $\Omega $, $%
\forall \alpha \in \mathbb{Z}_{+}^{m},\exists k>0,\exists c>0,$ $\exists
\varepsilon _{0}\in \left] 0,1\right] ,\forall \varepsilon \leq \varepsilon
_{0},$
\begin{equation*}
\sup_{x\in K}\left\vert \partial ^{\alpha }f_{\varepsilon }\left( x\right)
\right\vert \leq c\exp \left( k\varepsilon ^{-\frac{1}{2\sigma -1}}\right) ,
\end{equation*}%
and $\mathcal{N}^{\sigma }\left( \Omega \right) $ is the space of $\left(
f_{\varepsilon }\right) _{\varepsilon }\in C^{\infty }\left( \Omega \right)
^{\left] 0,1\right] }$ satisfying for every compact $K$ of $\Omega $,$%
\forall \alpha \in \mathbb{Z}_{+}^{m},\forall k>0,\exists c>0,$ $\exists
\varepsilon _{0}\in \left] 0,1\right] ,\forall \varepsilon \leq \varepsilon
_{0},$
\begin{equation*}
\sup_{x\in K}\left\vert \partial ^{\alpha }f_{\varepsilon }\left( x\right)
\right\vert \leq c\exp \left( -k\varepsilon ^{-\frac{1}{2\sigma -1}}\right)
\end{equation*}%
The functor $\Omega \rightarrow \mathcal{G}^{\sigma }\left( \Omega \right) $
being a sheaf of differential algebras on $\mathbb{R}^{n},$ we show that $%
\mathcal{G}^{\sigma }\left( \Omega \right) $\ contains the space of Gevrey
ultradistributions of order $(3\sigma -1),$ and the following diagram of
embeddings is commutative

\begin{equation*}
\begin{array}{ccc}
E^{\sigma }(\Omega ) & \rightarrow & \mathcal{G}^{\sigma }(\Omega ) \\
& \searrow & \uparrow \\
&  & D_{3\sigma -1}^{\prime }(\Omega )%
\end{array}%
\end{equation*}

We then develop a Gevrey microlocal analysis adapted to these algebras in
the spirit of \cite{Hor}, \cite{Rod} and \cite{NPS}. The starting point of
the Gevrey microlocal analysis in the framework of the algebra $\mathcal{G}%
^{\sigma }(\Omega )$ consists first in introducing the algebra of regular
generalized Gevrey ultradistributions $\mathcal{G}^{\sigma ,\infty }(\Omega )
$\ and then to prove the following fundamental result
\begin{equation*}
\mathcal{G}^{\sigma ,\infty }(\Omega )\cap D_{3\sigma -1}^{\prime }(\Omega
)=E^{\sigma }(\Omega )
\end{equation*}%
The functor $\Omega \rightarrow \mathcal{G}^{\sigma ,\infty }\left( \Omega
\right) $ is a subsheaf of $\mathcal{G}^{\sigma }.$ This permits to define
the generalized Gevrey singular support and then, with the help of the
Fourier transform, the generalized Gevrey wave front of $f\in \mathcal{G}%
^{\sigma }\left( \Omega \right) $, denoted $WF_{g}^{\sigma }\left( f\right) $%
, and further to give its main properties, as $WF_{g}^{\sigma }\left(
T\right) =WF^{\sigma }\left( T\right) ,$ if $T\in D_{3\sigma -1}^{\prime
}\left( \Omega \right) \cap \mathcal{G}^{\sigma }\left( \Omega \right) ,$
and
\begin{equation*}
WF_{g}^{\sigma }\left( P\left( x,D\right) f\right) \subset WF_{g}^{\sigma
}\left( f\right) ,\forall f\in \mathcal{G}^{\sigma }\left( \Omega \right) ,
\end{equation*}%
if $P\left( x,D\right) =\sum\limits_{\left\vert \alpha \right\vert \leq
m}a_{\alpha }\left( x\right) D^{\alpha }$ is a partial differential operator
with $\mathcal{G}^{\sigma ,\infty }\left( \Omega \right) $ coefficients.

Let us note that in \cite{Ben-Bouz 2}, the authors introduced a general well
adapted local and microlocal ultraregular analysis whitin Colombeau algebra $%
\mathcal{G}\left( \Omega \right) $.

Finally, we give an application of the introduced generalized Gevrey
microlocal analysis. The product of two generalized Gevrey
ultradistributions always exists, but there is no final description of the
generalized wave front of this product. Such problem is also still posed in
the Colombeau algebra. In \cite{Hor-Kun}, the well-known H\"{o}rmander's
result on the wave front of the product of two distributions, has been
extended to the case of two Colombeau generalized functions. We show this
result in the case of two generalized Gevrey ultradistributions, namely we
obtain the following result : let $f,g\in \mathcal{G}^{\sigma }\left( \Omega
\right) $, satisfying $\forall x\in \Omega ,$%
\begin{equation*}
\left( x,0\right) \notin WF_{g}^{\sigma }\left( f\right) +WF_{g}^{\sigma
}\left( g\right) ,
\end{equation*}%
then
\begin{equation*}
WF_{g}^{\sigma }\left( fg\right) \subseteq \left( WF_{g}^{\sigma }\left(
f\right) +WF_{g}^{\sigma }\left( g\right) \right) \cup WF_{g}^{\sigma
}\left( f\right) \cup WF_{g}^{\sigma }\left( g\right)
\end{equation*}

\section{Generalized Gevrey ultradistributions}

To define the algebra of generalized Gevrey ultradistributions, we first
introduce the algebra of moderate elements and its ideal of null elements
depending on the Gevrey order $\sigma \geq 1.$ The set $\Omega $\ is a non
void open of $\mathbb{R}^{n}.$

\begin{definition}
The space of moderate elements, denoted $\mathcal{E}_{m}^{\sigma }\left(
\Omega \right) ,$ is the space of $\left( f_{\varepsilon }\right)
_{\varepsilon }\in C^{\infty }\left( \Omega \right) ^{\left] 0,1\right] }$
satisfying for every compact $K$ of $\Omega $, $\forall \alpha \in \mathbb{Z}%
_{+}^{m},\exists k>0,$ $\exists c>0,\exists \varepsilon _{0}\in \left] 0,1%
\right] ,\forall \varepsilon \leq \varepsilon _{0}$,
\begin{equation}
\sup_{x\in K}\left\vert \partial ^{\alpha }f_{\varepsilon }\left( x\right)
\right\vert \leq c\exp \left( k\varepsilon ^{-\frac{1}{2\sigma -1}}\right)
\label{1*1}
\end{equation}%
The space of null elements, denoted $\mathcal{N}^{\sigma }\left( \Omega
\right) ,$ is the space of $\left( f_{\varepsilon }\right) _{\varepsilon
}\in C^{\infty }\left( \Omega \right) ^{\left] 0,1\right] }$ satisfying for
every compact $K$ of $\Omega $,$\forall \alpha \in \mathbb{Z}%
_{+}^{m},\forall k>0,\exists c>0,$ $\exists \varepsilon _{0}\in \left] 0,1%
\right] ,$ $\forall \varepsilon \leq \varepsilon _{0},$
\begin{equation}
\sup_{x\in K}\left\vert \partial ^{\alpha }f_{\varepsilon }\left( x\right)
\right\vert \leq c\exp \left( -k\varepsilon ^{-\frac{1}{2\sigma -1}}\right)
\label{1*2}
\end{equation}
\end{definition}

The main properties of the spaces $\mathcal{E}_{m}^{\sigma }\left( \Omega
\right) $ and $\mathcal{N}^{\sigma }\left( \Omega \right) $\ are given in
the following proposition.

\begin{proposition}
1) The space of moderate elements $\mathcal{E}_{m}^{\sigma }\left( \Omega
\right) $ is an algebra stable by derivation.

2) The space $\mathcal{N}^{\sigma }\left( \Omega \right) $ is an ideal of $%
\mathcal{E}_{m}^{\sigma }\left( \Omega \right) .$
\end{proposition}

\begin{proof}
1) Let $\left( f_{\varepsilon }\right) _{\varepsilon },\left( g_{\varepsilon
}\right) _{\varepsilon }\in \mathcal{E}_{m}^{\sigma }\left( \Omega \right) $
and $K$ be a compact of $\Omega $, then $\forall \beta \in \mathbb{Z}%
_{+}^{m},$ $\exists k_{1}=k_{1}\left( \beta \right) >0,\exists
c_{1}=c_{1}\left( \beta \right) >0,\exists \varepsilon _{1\beta }\in \left]
0,1\right] ,$ $\forall \varepsilon \leq \varepsilon _{1\beta },$%
\begin{equation}
\sup_{x\in K}\left\vert \partial ^{\beta }f_{\varepsilon }\left( x\right)
\right\vert \leq c_{1}\exp \left( k_{1}\varepsilon ^{-\frac{1}{2\sigma -1}%
}\right) ,  \label{4}
\end{equation}%
$\forall \beta \in \mathbb{Z}_{+}^{m},$ $\exists k_{2}=k_{2}\left( \beta
\right) >0,\exists c_{2}=c_{2}\left( \beta \right) >0,\exists \varepsilon
_{2\beta }\in \left] 0,1\right] ,\forall \varepsilon \leq \varepsilon
_{2\beta },$%
\begin{equation}
\sup_{x\in K}\left\vert \partial ^{\beta }g_{\varepsilon }\left( x\right)
\right\vert \leq c_{2}\exp \left( k_{2}\varepsilon ^{-\frac{1}{2\sigma -1}%
}\right)  \label{5}
\end{equation}%
Let $\alpha \in \mathbb{Z}_{+}^{m},$ then
\begin{equation*}
\left\vert \partial ^{\alpha }\left( f_{\varepsilon }g_{\varepsilon }\right)
\left( x\right) \right\vert \leq \sum_{\beta =0}^{\alpha }\binom{\alpha }{%
\beta }\left\vert \partial ^{\alpha -\beta }f_{\varepsilon }\left( x\right)
\right\vert \left\vert \partial ^{\beta }g_{\varepsilon }\left( x\right)
\right\vert
\end{equation*}%
For $k=\max \left\{ k_{1}\left( \beta \right) :\beta \leq \alpha \right\}
+\max \left\{ k_{2}\left( \beta \right) :\beta \leq \alpha \right\}
,\varepsilon \leq \min \left\{ \varepsilon _{1\beta },\varepsilon _{2\beta
};\left\vert \beta \right\vert \leq \left\vert \alpha \right\vert \right\} $
and $x\in K,$ we have
\begin{eqnarray*}
\exp \left( -k\varepsilon ^{-\frac{1}{2\sigma -1}}\right) \left\vert
\partial ^{\alpha }\left( f_{\varepsilon }g_{\varepsilon }\right) \left(
x\right) \right\vert &\leq &\sum_{\beta =0}^{\alpha }\binom{\alpha }{\beta }%
\exp \left( -k_{1}\varepsilon ^{-\frac{1}{2\sigma -1}}\right) \left\vert
\partial ^{\alpha -\beta }f_{\varepsilon }\left( x\right) \right\vert \\
&&\times \exp \left( -k_{2}\varepsilon ^{-\frac{1}{2\sigma -1}}\right)
\left\vert \partial ^{\beta }g_{\varepsilon }\left( x\right) \right\vert \\
&\leq &\sum_{\beta =0}^{\alpha }\binom{\alpha }{\beta }c_{1}\left( \alpha
-\beta \right) c_{2}\left( \beta \right) =c\left( \alpha \right) ,
\end{eqnarray*}%
i.e. $\left( f_{\varepsilon }g_{\varepsilon }\right) _{\varepsilon }\in
\mathcal{E}_{m}^{\sigma }\left( \Omega \right) $.

It is clear, from (\ref{4}) that for every compact $K$ of $\Omega $, $%
\forall \beta \in \mathbb{Z}_{+}^{m},$ $\exists k_{1}=k_{1}\left( \beta
+1\right) >0,\exists c_{1}=c_{1}\left( \beta +1\right) >0,\exists
\varepsilon _{1\beta }\in \left] 0,1\right] $ such that $\forall x\in
K,\forall \varepsilon \leq \varepsilon _{1\beta },$%
\begin{equation*}
\left\vert \partial ^{\beta }\left( \partial f_{\varepsilon }\right) \left(
x\right) \right\vert \leq c_{1}\exp \left( k_{1}\varepsilon ^{-\frac{1}{%
2\sigma -1}}\right) ,
\end{equation*}%
i.e. $\left( \partial f_{\varepsilon }\right) _{\varepsilon }\in \mathcal{E}%
_{m}^{\sigma }\left( \Omega \right) .$

2) If $\left( g_{\varepsilon }\right) _{\varepsilon }\in \mathcal{N}^{\sigma
}\left( \Omega \right) ,$ for every $K$ compact of $\Omega ,$ $\forall \beta
\in \mathbb{Z}_{+}^{m},\forall k_{2}>0,\exists c_{2}=c_{2}\left( \beta
,k_{2}\right) >0,\exists \varepsilon _{2\beta }\in \left] 0,1\right] ,$%
\begin{equation*}
\left\vert \partial ^{\alpha }g_{\varepsilon }\left( x\right) \right\vert
\leq c_{2}\exp \left( -k_{2}\varepsilon ^{-\frac{1}{2\sigma -1}}\right)
,\forall x\in K,\forall \varepsilon \leq \varepsilon _{2\beta }
\end{equation*}

Let $\alpha \in \mathbb{Z}_{+}^{m}$ and $k>0,$ then
\begin{eqnarray*}
\exp \left( k\varepsilon ^{-\frac{1}{2\sigma -1}}\right) \left\vert \partial
^{\alpha }\left( f_{\varepsilon }g_{\varepsilon }\right) \left( x\right)
\right\vert &\leq &\exp \left( k\varepsilon ^{-\frac{1}{2\sigma -1}}\right)
\sum_{\beta =0}^{\alpha }\binom{\alpha }{\beta }\left\vert \partial ^{\alpha
-\beta }f_{\varepsilon }\left( x\right) \right\vert \times \\
&&\times \left\vert \partial ^{\beta }g_{\varepsilon }\left( x\right)
\right\vert
\end{eqnarray*}%
Let $k_{2}=\max \left\{ k_{1}\left( \beta \right) ;\beta \leq \alpha
\right\} +k$ and $\varepsilon \leq \min \left\{ \varepsilon _{1\beta
},\varepsilon _{2\beta };\beta \leq \alpha \right\} ,$ then $\forall x\in K$%
,
\begin{eqnarray*}
\exp \left( k\varepsilon ^{-\frac{1}{2\sigma -1}}\right) \left\vert \partial
^{\alpha }\left( f_{\varepsilon }g_{\varepsilon }\right) \left( x\right)
\right\vert &\leq &\sum_{\beta =0}^{\alpha }\binom{\alpha }{\beta }\left[
\exp \left( -k_{1}\varepsilon ^{-\frac{1}{2\sigma -1}}\right) \left\vert
\partial ^{\alpha -\beta }f_{\varepsilon }\left( x\right) \right\vert \right.
\\
&&\times \left. \exp \left( k_{2}\varepsilon ^{-\frac{1}{2\sigma -1}}\right)
\left\vert \partial ^{\beta }g_{\varepsilon }\left( x\right) \right\vert %
\right] \\
&\leq &\sum_{\beta =0}^{\alpha }\binom{\alpha }{\beta }c_{1}\left( \alpha
-\beta \right) c_{2}\left( \beta ,k_{2}\right) =c\left( \alpha ,k\right) ,
\end{eqnarray*}%
which shows that $\left( f_{\varepsilon }g_{\varepsilon }\right)
_{\varepsilon }\in \mathcal{N}^{\sigma }\left( \Omega \right) $
\end{proof}

\begin{remark}
The algebra of moderate elements $\mathcal{E}_{m}^{\sigma }\left( \Omega
\right) $\ is not necessary stable by $\sigma -$ultradifferentiable
operators, because the constant $c$ in $\left( \ref{1*1}\right) $ dependents
of $\alpha .$
\end{remark}

According to the topological construction of Colombeau type algebras of
generalized functions, we introduce the desired algebras.

\begin{definition}
The algebra of generalized Gevrey ultradistributions of order $\sigma \geq 1$%
, denoted $\mathcal{G}^{\sigma }\left( \Omega \right) ,$ is the quotient
algebra
\begin{equation*}
\mathcal{G}^{\sigma }\left( \Omega \right) =\frac{\mathcal{E}_{m}^{\sigma
}\left( \Omega \right) }{\mathcal{N}^{\sigma }\left( \Omega \right) }
\end{equation*}
\end{definition}

A comparison of the structure of our algebras $\mathcal{G}^{\sigma }\left(
\Omega \right) $\ and the Colombeau algebra $\mathcal{G}\left( \Omega
\right) $ is given in the following remark.

\begin{remark}
The Colombeau algebra $\mathcal{G}\left( \Omega \right) :=\frac{\mathcal{E}%
_{m}\left( \Omega \right) }{\mathcal{N}\left( \Omega \right) },$ where $%
\mathcal{E}_{m}(\Omega )$ is the space of $\left( f_{\varepsilon }\right)
_{\varepsilon }\in C^{\infty }\left( \Omega \right) ^{\left] 0,1\right] }$
satisfying for every compact $K$ of $\Omega $, $\forall \alpha \in \mathbb{Z}%
_{+}^{m},\exists k>0,\exists c>0,$ $\exists \varepsilon _{0}\in \left] 0,1%
\right] ,\forall \varepsilon \leq \varepsilon _{0},$
\begin{equation*}
\sup_{x\in K}\left\vert \partial ^{\alpha }f_{\varepsilon }\left( x\right)
\right\vert \leq c\varepsilon ^{-k},
\end{equation*}%
and $\mathcal{N}\left( \Omega \right) $ is the space of $\left(
f_{\varepsilon }\right) _{\varepsilon }\in C^{\infty }\left( \Omega \right)
^{\left] 0,1\right] }$ satisfying for every compact $K$ of $\Omega $,$%
\forall \alpha \in \mathbb{Z}_{+}^{m},\forall k>0,\exists c>0,$ $\exists
\varepsilon _{0}\in \left] 0,1\right] ,\forall \varepsilon \leq \varepsilon
_{0},$
\begin{equation*}
\sup_{x\in K}\left\vert \partial ^{\alpha }f_{\varepsilon }\left( x\right)
\right\vert \leq c\varepsilon ^{k}
\end{equation*}%
Due to the inequality
\begin{equation*}
\exp \left( -\varepsilon ^{-\frac{1}{2\sigma -1}}\right) \leq \varepsilon
,\forall \varepsilon \in \left] 0,1\right] ,
\end{equation*}%
we have the strict inclusions $\mathcal{N}^{\sigma }\left( \Omega \right)
\subset \mathcal{N}^{\tau }\left( \Omega \right) \subset \mathcal{N}\left(
\Omega \right) \subset \mathcal{E}_{m}\left( \Omega \right) \subset \mathcal{%
E}_{m}^{\tau }\left( \Omega \right) \subset \mathcal{E}_{m}^{\sigma }\left(
\Omega \right) ,$ with $\sigma <\tau $.
\end{remark}

We have the null characterization of the ideal $\mathcal{N}^{\sigma }\left(
\Omega \right) $.

\begin{proposition}
\label{carN}Let $\left( u_{\epsilon }\right) _{\epsilon }\in \mathcal{E}%
_{m}^{\sigma }\left( \Omega \right) ,$ then $\left( u_{\epsilon }\right)
_{\epsilon }\in \mathcal{N}^{\sigma }\left( \Omega \right) $ if and only if
for every compact $K$ of $\Omega $, $\forall k>0,\exists c>0,$ $\exists
\varepsilon _{0}\in \left] 0,1\right] ,$ $\forall \varepsilon \leq
\varepsilon _{0},$
\begin{equation}
\sup_{x\in K}\left\vert f_{\varepsilon }\left( x\right) \right\vert \leq
c\exp \left( -k\varepsilon ^{-\frac{1}{2\sigma -1}}\right)  \label{1*6}
\end{equation}
\end{proposition}

\begin{proof}
Let $\left( u_{\epsilon }\right) _{\varepsilon }\in \mathcal{E}_{m}^{\sigma
}\left( \Omega \right) $ satisfying $\left( \ref{1*6}\right) $, we will show
that $\left( \partial _{i}u_{\epsilon }\right) _{\epsilon }$ satisfy $\left( %
\ref{1*6}\right) $ when $i=1,..,m$, and then it will follow by induction
that $\left( u_{\epsilon }\right) _{\epsilon }\in \mathcal{N}^{\sigma
}\left( \Omega \right) .$

Suppose that $u_{\epsilon }$ has a real values, in the complex case we do
the calculus separately for the real and imaginary part of $u_{\epsilon }$.
Let $K$ be a compact of $\Omega ,$ for $\delta =\min \left( 1,dist\left(
K,\partial \Omega \right) \right) $, set $L=K+\overline{B\left( 0,\frac{%
\delta }{2}\right) },$ then $K\subset \subset L\subset \subset \Omega .$ By
the moderateness of $\left( u_{\epsilon }\right) _{\varepsilon }$, we have $%
\exists k_{1}>0,\exists c_{1}>0,$ $\exists \varepsilon _{1}\in \left] 0,1%
\right] ,\forall \varepsilon \leq \varepsilon _{1}$%
\begin{equation}
\underset{x\in L}{\sup }\left\vert \partial _{i}^{2}u_{\epsilon }\left(
x\right) \right\vert \leq c_{1}\exp \left( k_{1}\varepsilon ^{-\frac{1}{%
2\sigma -1}}\right)  \label{1*7}
\end{equation}%
By the assumption $\left( \ref{1*6}\right) $, $\forall k>0,\exists c_{2}>0,$
$\exists \varepsilon _{2}\in \left] 0,1\right] ,\forall \varepsilon \leq
\varepsilon _{2}$%
\begin{equation}
\underset{x\in L}{\sup }\left\vert u_{\epsilon }\left( x\right) \right\vert
\leq c_{2}\exp \left( -\left( 2k+k_{1}\right) \varepsilon ^{-\frac{1}{%
2\sigma -1}}\right)  \label{1*8}
\end{equation}%
Let $x\in K,\varepsilon $ sufficiently small and $r=\exp \left( -\left(
k+k_{1}\right) \varepsilon ^{-\frac{1}{2\sigma -1}}\right) <\frac{\delta }{2}
$. By Taylor's formula, we have
\begin{equation*}
\partial _{i}u_{\epsilon }\left( x\right) =\frac{\left( u_{\epsilon }\left(
x+re_{i}\right) -u_{\epsilon }\left( x\right) \right) }{r}-\frac{1}{2}%
\partial _{i}^{2}u_{\epsilon }\left( x+\theta re_{i}\right) r,
\end{equation*}%
where $e_{i}$ is i$^{th}$ vector of the canonical base of $\mathbb{R}^{m},$
hence $\left( x+\theta re_{i}\right) \in L,$ and then
\begin{equation*}
\left\vert \partial _{i}u_{\epsilon }\left( x\right) \right\vert \leq
\left\vert u_{\epsilon }\left( x+re_{i}\right) -u_{\epsilon }\left( x\right)
\right\vert r^{-1}+\frac{1}{2}\left\vert \partial _{i}^{2}u_{\epsilon
}\left( x+\theta re_{i}\right) \right\vert r
\end{equation*}%
From $\left( \ref{1*7}\right) $ and $\left( \ref{1*8}\right) :$ $\left\vert
u_{\epsilon }\left( x+re_{i}\right) -u_{\epsilon }\left( x\right)
\right\vert r^{-1}\leq c_{2}\exp \left( -k\varepsilon ^{-\frac{1}{2\sigma -1}%
}\right) $ and $\left\vert \partial _{i}^{2}u_{\epsilon }\left( x+\theta
re_{i}\right) \right\vert r\leq c_{1}\exp \left( -k\varepsilon ^{-\frac{1}{%
2\sigma -1}}\right) ,$ so
\begin{equation*}
\left\vert \partial _{i}u_{\epsilon }\left( x\right) \right\vert \leq c\exp
\left( -k\varepsilon ^{-\frac{1}{2\sigma -1}}\right) ,
\end{equation*}%
which gives the proof.
\end{proof}

\begin{proposition}
If $P$ is a polynomial function and $f=cl\left( f_{\varepsilon }\right)
_{\varepsilon }\in \mathcal{G}^{\sigma }\left( \Omega \right) ,$ then $%
P\left( f\right) =\left( P\left( f_{\varepsilon }\right) \right)
_{\varepsilon }+\mathcal{N}^{\sigma }\left( \Omega \right) $ is well defined
element of $\mathcal{G}^{\sigma }\left( \Omega \right) $.
\end{proposition}

\begin{proof}
Let $\left( f_{\varepsilon }\right) _{\varepsilon }\in \mathcal{E}%
_{m}^{\sigma }\left( \Omega \right) $, $P\left( \xi \right)
=\sum\limits_{\left\vert \alpha \right\vert \leq m}a_{\alpha }\xi ^{\alpha }$
and $K$ be a compact of $\Omega $, then we have $\forall \alpha \in \mathbb{Z%
}_{+}^{m},\exists k=k\left( \alpha \right) >0,$ $\exists c=c\left( \alpha
\right) >0,\exists \varepsilon _{0}=\varepsilon \left( \alpha \right) \in %
\left] 0,1\right] ,\forall \varepsilon \leq \varepsilon _{0}$,
\begin{equation}
\sup_{x\in K}\left\vert \partial ^{\alpha }f_{\varepsilon }\left( x\right)
\right\vert \leq c\exp \left( k\varepsilon ^{-\frac{1}{2\sigma -1}}\right)
\label{com1}
\end{equation}%
Let $\beta \in \mathbb{Z}_{+}^{m}$, so
\begin{equation*}
\left\vert \partial ^{\beta }P\left( f_{\varepsilon }\right) \left( x\right)
\right\vert \leq \sum\limits_{\left\vert \alpha \right\vert \leq
m}\left\vert a_{\alpha }\right\vert \left\vert \partial ^{\beta
}f_{\varepsilon }^{\alpha }\left( x\right) \right\vert ,
\end{equation*}%
by Leibniz formula and (\ref{com1}), we obtain
\begin{equation*}
\left\vert \partial ^{\beta }P\left( f_{\varepsilon }\right) \left( x\right)
\right\vert \leq \sum\limits_{\substack{ \left\vert \alpha \right\vert \leq
m  \\ \gamma \leq \beta }}c_{\alpha ,\gamma }\left( \exp \left( k_{\alpha
,\gamma }\varepsilon ^{-\frac{1}{2\sigma -1}}\right) \right) ^{n_{\alpha
,\gamma }},
\end{equation*}%
where $c_{\alpha ,\gamma }>0$ and $n_{\alpha ,\gamma }\in \mathbb{Z}_{+}$.
Hence
\begin{equation*}
\left\vert \partial ^{\beta }P\left( f_{\varepsilon }\right) \left( x\right)
\right\vert \leq c\exp \left( k\varepsilon ^{-\frac{1}{2\sigma -1}}\right)
\end{equation*}%
One can easily cheek that if $\left( f_{\varepsilon }\right) _{\varepsilon
}\in \mathcal{N}^{\sigma }\left( \Omega \right) ,$ then $\left( P\left(
f_{\varepsilon }\right) \right) _{\varepsilon }\in \mathcal{N}^{\sigma
}\left( \Omega \right) $
\end{proof}

The space of functions slowly increasing, denoted $\mathcal{O}_{M}\left(
\mathbb{K}^{m}\right) ,$ is the space of $C^{\infty }$-functions all
derivatives growing at most like some power of $\left\vert x\right\vert ,$
as $\left\vert x\right\vert \rightarrow +\infty $, where $\mathbb{K}%
^{m}\simeq \mathbb{R}^{m}$ or $\mathbb{R}^{2m}.$

\begin{corollary}
If $v\in \mathcal{O}_{M}\left( \mathbb{K}^{m}\right) $ and $f=\left(
f_{1},f_{2},...,f_{m}\right) \in \mathcal{G}^{\sigma }\left( \Omega \right)
^{m},$ then $v\circ f:=\left( v\circ f_{\varepsilon }\right) _{\varepsilon }+%
\mathcal{N}^{\sigma }\left( \Omega \right) $ is a well defined element of $%
\mathcal{G}^{\sigma }\left( \Omega \right) .$
\end{corollary}

\section{Generalized point values}

The ring of Gevrey generalized complex numbers, denoted $\mathcal{C}^{\sigma
},$ is defined by the quotient
\begin{equation*}
\mathcal{C}^{\sigma }=\frac{\mathcal{E}_{0}^{\sigma }}{\mathcal{N}%
_{0}^{\sigma }}\text{ \ ,}
\end{equation*}%
where%
\begin{equation*}
\begin{array}{c}
\mathcal{E}_{0}^{\sigma }=\left\{ \left( a_{\varepsilon }\right)
_{\varepsilon }\in \mathbb{C}^{\left] 0,1\right] };\exists k>0,\exists
c>0,\exists \varepsilon _{0}\in \left] 0,1\right] \text{, such that \ \ \ \ }%
\right. \\
\multicolumn{1}{r}{\left. \forall \varepsilon \leq \varepsilon
_{0},\left\vert a_{\varepsilon }\right\vert \leq c\exp \left( k\varepsilon
^{-\frac{1}{2\sigma -1}}\right) \right\}}%
\end{array}%
\end{equation*}
and%
\begin{equation*}
\begin{array}{r}
\mathcal{N}_{0}^{\sigma }=\left\{ \left( a_{\varepsilon }\right)
_{\varepsilon }\in \mathbb{C}^{\left] 0,1\right] };\forall k>0,\exists
c>0,\exists \varepsilon _{0}\in \left] 0,1\right] \text{, such that \ \ }%
\right. \\
\left. \forall \varepsilon \leq \varepsilon _{0},\left\vert a_{\varepsilon
}\right\vert \leq c\exp \left( -k\varepsilon ^{-\frac{1}{2\sigma -1}}\right)
\right\}%
\end{array}%
\end{equation*}%
It is not difficult to see that $\mathcal{E}_{0}^{\sigma }$ is an algebra
and $\mathcal{N}_{0}^{\sigma }$ is an ideal of $\mathcal{E}_{0}^{\sigma }$.
The ring $\mathcal{C}^{\sigma }$ motivates the following, easy to prove,
result.

\begin{proposition}
If $u\in \mathcal{G}^{\sigma }\left( \Omega \right) $ and $x\in \Omega $,
then the element $u\left( x\right) $ represented by $\left( u_{\varepsilon
}\left( x\right) \right) _{\varepsilon }$ is an element of $\mathcal{C}%
^{\sigma }$ independent of the representative$\ \left( u_{\varepsilon
}\right) _{\varepsilon }$ of $u.$
\end{proposition}

A generalized Gevrey ultradistribution is not defined by their point values,
we give here an example of generalized Gevrey ultradistribution $f=\left[
\left( f_{\varepsilon }\right) _{\varepsilon }\right] \notin \mathcal{N}%
^{\sigma }\left( \mathbb{R}\right) $, but $\left[ \left( f_{\varepsilon
}\left( x\right) \right) _{\varepsilon }\right] \in \mathcal{N}_{0}^{\sigma
} $ for every $x\in \mathbb{R}.$ Let $\varphi \in D\left( \mathbb{R}\right) $
such that $\varphi \left( 0\right) \neq 0$. For $\varepsilon \in \left] 0,1%
\right] ,$ define
\begin{equation*}
f_{\varepsilon }\left( x\right) =x\exp \left( -\varepsilon ^{-\frac{1}{%
2\sigma -1}}\right) \varphi \left( \frac{x}{\varepsilon }\right) ,x\in
\mathbb{R}
\end{equation*}%
It is clear that $\left( f_{\varepsilon }\right) _{\varepsilon }\in \mathcal{%
E}_{m}^{\sigma }\left( \mathbb{R}\right) $. Let $K$ be a compact
neighborhood of $0$, then
\begin{equation*}
\sup_{K}\left\vert f^{\prime }\left( x\right) \right\vert \geq \left\vert
f_{\varepsilon }^{\prime }\left( 0\right) \right\vert =\exp \left(
-\varepsilon ^{-\frac{1}{2\sigma -1}}\right) \left\vert \varphi \left(
0\right) \right\vert ,
\end{equation*}%
which show that $\left( f_{\varepsilon }\right) _{\varepsilon }\notin
\mathcal{N}^{\sigma }\left( \mathbb{R}\right) $. For any $x_{0}\in \mathbb{R}%
,$ there exists $\varepsilon _{0}$ such that $\varphi \left( \frac{x_{0}}{%
\varepsilon }\right) =0,\forall \varepsilon \leq \varepsilon _{0},$ i.e. $%
f\left( x_{0}\right) \in \mathcal{N}_{0}^{\sigma }.$

In order to give a solution to this situation,\ set
\begin{equation}
\Omega _{M}^{\sigma }=\left\{ \left( x_{\varepsilon }\right) _{\varepsilon
}\in \Omega ^{\left] 0,1\right] }:\exists k>0,\exists c>0,\exists
\varepsilon _{0}>0,\forall \varepsilon \leq \varepsilon _{0},\left\vert
x_{\varepsilon }\right\vert \leq ce^{k\varepsilon ^{-\frac{1}{2\sigma -1}%
}}\right\}
\end{equation}%
Define in $\Omega _{M}^{\sigma }$ the equivalence relation by
\begin{equation}
x_{\varepsilon }\sim y_{\varepsilon }\Longleftrightarrow \forall k>0,\exists
c>0,\exists \varepsilon _{0}>0,\forall \varepsilon \leq \varepsilon
_{0},\left\vert x_{\varepsilon }-y_{\varepsilon }\right\vert \leq
ce^{-k\varepsilon ^{-\frac{1}{2\sigma -1}}}
\end{equation}

\begin{definition}
The set $\widetilde{\Omega }^{\sigma }=\Omega _{M}^{\sigma }/\sim $ is
called the set of generalized Gevrey points. The set of compactly supported
Gevrey points is defined by
\begin{equation}
\widetilde{\Omega }_{c}^{\sigma }=\left\{ \widetilde{x}=\left[ \left(
x_{\varepsilon }\right) _{\varepsilon }\right] \in \widetilde{\Omega }%
^{\sigma }:\exists K\text{ a compact set of }\Omega ,\exists \varepsilon
_{0}>0,\forall \varepsilon \leq \varepsilon _{0},x_{\varepsilon }\in
K\right\}
\end{equation}
\end{definition}

\begin{remark}
It is easy to see that $\widetilde{\Omega }_{c}^{\sigma }$-property does not
depend on the choice of the representative.
\end{remark}

\begin{proposition}
Let $f\in \mathcal{G}^{\sigma }\left( \Omega \right) $ and $\widetilde{x}=%
\left[ \left( x_{\varepsilon }\right) _{\varepsilon }\right] \in \widetilde{%
\Omega }_{c}^{\sigma },$ then the generalized Gevrey point value of $f$ at $%
\widetilde{x}$, i.e.
\begin{equation*}
f\left( \widetilde{x}\right) =\left[ \left( f_{\varepsilon }\left(
x_{\varepsilon }\right) \right) _{\varepsilon }\right]
\end{equation*}
is a well-defined element of the algebra of generalized Gevrey complex
numbers $\mathcal{C}^{\sigma }.$
\end{proposition}

\begin{proof}
Let $f\in \mathcal{G}^{\sigma }\left( \Omega \right) $ and $\widetilde{x}=%
\left[ \left( x_{\varepsilon }\right) _{\varepsilon }\right] \in \widetilde{%
\Omega }_{c}^{\sigma }$, there exists a compact $K$ of $\Omega $ such that $%
x_{\varepsilon }\in K$ for $\varepsilon $ small, then $\exists k>0,\exists
c>0,\exists \varepsilon _{0}>0,\forall \varepsilon \leq \varepsilon _{0},$%
\begin{equation*}
\left\vert f_{\varepsilon }\left( x_{\varepsilon }\right) \right\vert \leq
\sup_{x\in K}\left\vert f_{\varepsilon }\left( x\right) \right\vert \leq
c\exp \left( k\varepsilon ^{-\frac{1}{2\sigma -1}}\right)
\end{equation*}%
Therefore $\left( f_{\varepsilon }\left( x_{\varepsilon }\right) \right)
_{\varepsilon }\in \mathcal{E}_{0}^{\sigma }$, and it is clear that if $f\in
\mathcal{N}^{\sigma }\left( \Omega \right) ,$ then $\left( f_{\varepsilon
}\left( x_{\varepsilon }\right) \right) _{\varepsilon }\in \mathcal{N}%
_{0}^{\sigma }$, i.e. $f\left( \widetilde{x}\right) $ does not depend on the
choice of the representative $\left( f_{\varepsilon }\right) _{\varepsilon }$%
.

Let now $\widetilde{x}=\left[ \left( x_{\varepsilon }\right) _{\varepsilon }%
\right] \sim \widetilde{y}=\left[ \left( y_{\varepsilon }\right)
_{\varepsilon }\right] $, then $\forall k>0,\exists c>0,\exists \varepsilon
_{0}>0,\forall \varepsilon \leq \varepsilon _{0},$%
\begin{equation*}
\left\vert x_{\varepsilon }-y_{\varepsilon }\right\vert \leq c\exp \left(
-k\varepsilon ^{-\frac{1}{2\sigma -1}}\right)
\end{equation*}%
Since $\left( f_{\varepsilon }\right) _{\varepsilon }\in \mathcal{E}^{\sigma
}\left( \Omega \right) ,$ so for every compact $K$ of $\Omega $, $\forall
j\in \left\{ 1,m\right\} ,\exists k_{j}>0,\exists c_{j}>0,\exists
\varepsilon _{j}>0,\forall \varepsilon \leq \varepsilon _{j},$%
\begin{equation*}
\sup_{x\in K}\left\vert \frac{\partial }{\partial x_{j}}f_{\varepsilon
}\left( x\right) \right\vert \leq c_{j}\exp \left( k_{j}\varepsilon ^{-\frac{%
1}{2\sigma -1}}\right)
\end{equation*}%
We have
\begin{equation*}
\left\vert f_{\varepsilon }\left( x_{\varepsilon }\right) -f_{\varepsilon
}\left( y_{\varepsilon }\right) \right\vert \leq \left\vert x_{\varepsilon
}-y_{\varepsilon }\right\vert \sum_{j=1}^{m}\int_{0}^{1}\left\vert \left(
\frac{\partial }{\partial x_{j}}f_{\varepsilon }\right) \left(
x_{\varepsilon }+t\left( y_{\varepsilon }-x_{\varepsilon }\right) \right)
\right\vert dt,
\end{equation*}%
and $x_{\varepsilon }+t\left( y_{\varepsilon }-x_{\varepsilon }\right) $
remains within some compact $K$ of $\Omega $ for $\varepsilon \leq
\varepsilon ^{\prime }$. Let $k^{\prime }>0$, then for $k+k^{\prime
}=\sup\limits_{j}k_{j}$ and $\varepsilon \leq \min \left( \varepsilon
^{\prime },\varepsilon _{0},\varepsilon _{j}:i=1,m\right) ,$ we have
\begin{equation*}
\left\vert f_{\varepsilon }\left( x_{\varepsilon }\right) -f_{\varepsilon
}\left( y_{\varepsilon }\right) \right\vert \leq c\exp \left( -k^{\prime
}\varepsilon ^{-\frac{1}{2\sigma -1}}\right) \text{,}
\end{equation*}%
which gives $\left( f_{\varepsilon }\left( x_{\varepsilon }\right)
-f_{\varepsilon }\left( y_{\varepsilon }\right) \right) _{\varepsilon }\in
\mathcal{N}_{0}^{\sigma }.$
\end{proof}

The characterization of nullity of $f$ $\in \mathcal{G}^{\sigma }\left(
\Omega \right) $ is given by the following theorem.

\begin{theorem}
\label{th-car-p-v}Let $f\in \mathcal{G}^{\sigma }\left( \Omega \right) ,$
then
\begin{equation*}
f=0\text{ in }\mathcal{G}^{\sigma }\left( \Omega \right) \Longleftrightarrow
f\left( \widetilde{x}\right) =0\text{ in }\mathcal{C}^{\sigma }\text{ for
all }\widetilde{x}\in \widetilde{\Omega }_{c}^{\sigma }
\end{equation*}
\end{theorem}

\begin{proof}
It is easy to see that if $f\in \mathcal{N}^{\sigma }\left( \Omega \right) $
then $f\left( \widetilde{x}\right) \in \mathcal{N}_{0}^{\sigma },\forall
\widetilde{x}\in \widetilde{\Omega }_{c}^{\sigma }$. Suppose that $f\neq 0$
in $\mathcal{G}^{\sigma }\left( \Omega \right) ,$ then by the
characterization of $\mathcal{N}^{\sigma }\left( \Omega \right) $ we have,
there exists a compact $K$ of $\Omega ,\exists k>0,\forall c>0,\forall
\varepsilon _{0}>0,\exists \varepsilon \leq \varepsilon _{0},$
\begin{equation*}
\sup_{K}\left\vert f_{\varepsilon }\left( x\right) \right\vert >c\exp \left(
-k\varepsilon ^{-\frac{1}{2\sigma -1}}\right)
\end{equation*}%
So there exists a sequence $\varepsilon _{m}\searrow 0$ and $x_{m}\in K$
such that $\forall m\in \mathbb{Z}^{+},$
\begin{equation}
\left\vert f_{\varepsilon _{m}}\left( x_{m}\right) \right\vert >\exp \left(
-k\varepsilon _{m}^{-\frac{1}{2\sigma -1}}\right)  \label{1*9}
\end{equation}%
For $\varepsilon >0$ we set $x_{\varepsilon }=x_{m}$ when $\varepsilon
_{m+1}<\varepsilon \leq \varepsilon _{m}.$ We have $\left( x_{\varepsilon
}\right) _{\varepsilon }\in \Omega _{M}^{\sigma }$ with values in $K$, so $%
\widetilde{x}=\left[ \left( x_{\varepsilon }\right) _{\varepsilon }\right]
\in \widetilde{\Omega }_{c}^{\sigma }$ and (\ref{1*9}) means that$\ \left(
f_{\varepsilon }\left( x_{\varepsilon }\right) \right) _{\varepsilon }\notin
\mathcal{N}_{0}^{\sigma }$, i.e. $f\left( \widetilde{x}\right) \neq 0$ in $%
\mathcal{C}^{\sigma }.$
\end{proof}

\section{Embedding of Gevrey ultradistributions with compact support}

We recall some definitions and results on Gevrey ultradistributions, see
\cite{Kom}, \cite{LM} or \cite{Rod}.

\begin{definition}
A function $f$ $\in E^{\sigma }\left( \Omega \right) ,$ if $f\in C^{\infty
}\left( \Omega \right) $ and for every compact $K$ of $\Omega $, $\exists
c>0,\forall \alpha \in \mathbb{Z}_{+}^{m},$%
\begin{equation*}
\sup_{x\in K}\left| \partial ^{\alpha }f\left( x\right) \right| \leq
c^{\left| \alpha \right| +1}\left( \alpha !\right) ^{\sigma }
\end{equation*}
\end{definition}

Obviously we have $E^{t}\left( \Omega \right) \subset E^{\sigma }\left(
\Omega \right) $ if $1\leq t\leq \sigma .$ It is well known that $%
E^{1}\left( \Omega \right) =A\left( \Omega \right) $ is the space of all
real analytic functions in $\Omega $. Denote by $D^{\sigma }\left( \Omega
\right) $ the space $E^{\sigma }\left( \Omega \right) \cap C_{0}^{\infty
}\left( \Omega \right) ,$ then $D^{\sigma }\left( \Omega \right) $\ is non
trivial if and only if $\sigma >1.$ The topological dual of $D^{\sigma
}\left( \Omega \right) ,$ denoted $D_{\sigma }^{\prime }\left( \Omega
\right) ,$\ is called the space of Gevrey ultradistributions of order $%
\sigma .$\ The space $E_{\sigma }^{\prime }\left( \Omega \right) $ is the
topological dual of $E^{\sigma }\left( \Omega \right) $\ and is identified
with\ the space of Gevrey ultradistributions with compact support.

\begin{definition}
A differential operator of infinite order $P\left( D\right)
=\sum\limits_{\gamma \in \mathbb{Z}_{+}^{m}}a_{\gamma }D^{\gamma }$ is
called a $\sigma $-ultradifferential operator, if for every $h>0$ there
exist $c>0$ such that $\forall \gamma \in \mathbb{Z}_{+}^{m},$
\begin{equation}
\left\vert a_{\gamma }\right\vert \leq c\frac{h^{\left\vert \gamma
\right\vert }}{\left( \gamma !\right) ^{\sigma }}  \label{12}
\end{equation}
\end{definition}

The importance of $\sigma $-ultradifferential operators\ lies in the
following result.

\begin{proposition}
Let $T\in E_{\sigma }^{\prime }\left( \Omega \right) ,\sigma >1$ and $%
suppT\subset K,$ then there exist a $\sigma $-ultradifferential operator $%
P\left( D\right) =\sum\limits_{\gamma \in \mathbb{Z}_{+}^{m}}a_{\gamma
}D^{\gamma }$ , $M>0$\ and continuous functions $f_{\gamma }\in C_{0}\left(
K\right) $ such that $\sup\limits_{\gamma \in \mathbb{Z}_{+}^{m},x\in
K}\left\vert f_{\gamma }\left( x\right) \right\vert \leq M$ and
\begin{equation*}
T=\sum_{\gamma \in \mathbb{Z}_{+}^{m}}a_{\gamma }D^{\gamma }f_{\gamma }
\end{equation*}
\end{proposition}

The space $\mathcal{S}^{\left( \sigma \right) }\left( \mathbb{R}^{m}\right)
,\sigma >1,$ see \cite{Gel}, is the space of functions $\varphi \in
C^{\infty }\left( \mathbb{R}^{m}\right) $ such that $\forall b>0,$ we have
\begin{equation}
\left\| \varphi \right\| _{b,\sigma }=\sup_{\alpha ,\beta \in \mathbb{Z}%
_{+}^{m}}\int \frac{\left| x\right| ^{\left| \beta \right| }}{b^{\left|
\alpha +\beta \right| }\alpha !^{\sigma }\beta !^{\sigma }}\left| \partial
^{\alpha }\varphi \left( x\right) \right| dx<\infty  \label{2-1}
\end{equation}

\begin{lemma}
\label{lem3}There exists $\phi \in \mathcal{S}^{\left( \sigma \right)
}\left( \mathbb{R}^{m}\right) $ satisfying
\begin{equation*}
\int \phi \left( x\right) dx=1\text{ and }\int x^{\alpha }\phi \left(
x\right) dx=0,\forall \alpha \in \mathbb{Z}_{+}^{m}\backslash \{0\}
\end{equation*}
\end{lemma}

\begin{proof}
For an example of function $\phi \in \mathcal{S}^{\left( \sigma \right) }$
satisfying these conditions, take the\ Fourier transform of a function of
the class $D^{\left( \sigma \right) }\left( \mathbb{R}^{m}\right) $ equal $1$
in a neighborhood of the origin. Here $D^{\left( \sigma \right) }\left(
\mathbb{R}^{m}\right) $ denotes the projective Gevrey space of order $\sigma
,$ i.e. $D^{\left( \sigma \right) }\left( \mathbb{R}^{m}\right) =E^{(\sigma
)}\left( \mathbb{R}^{m}\right) \cap C_{0}^{\infty }\left( \mathbb{R}%
^{m}\right) $, where $f\in E^{(\sigma )}\left( \mathbb{R}^{m}\right) ,$\ if $%
f\in C^{\infty }\left( \mathbb{R}^{m}\right) $ and for every compact subset $%
K$ of $\mathbb{R}^{m},\forall b>0,\exists c>0,\forall \alpha \in \mathbb{Z}%
_{+}^{m},$%
\begin{equation}
\sup_{x\in K}\left| \partial ^{\alpha }f\left( x\right) \right| \leq
cb^{\left| \alpha \right| }\left( \alpha !\right) ^{\sigma }
\end{equation}
\end{proof}

\begin{definition}
The net $\phi _{\varepsilon }=\varepsilon ^{-m}\phi \left( ./\varepsilon
\right) ,\varepsilon \in \left] 0,1\right] ,$ where $\phi $\ satisfies the
conditions of lemma \ref{lem3}, is called a net of mollifiers.
\end{definition}

The space $E^{\sigma }\left( \Omega \right) $ is embedded into $\mathcal{G}%
^{\sigma }\left( \Omega \right) $ by the standard canonical injection
\begin{equation}
\begin{array}{llll}
I: & E^{\sigma }\left( \Omega \right) & \rightarrow & \mathcal{G}^{\sigma
}\left( \Omega \right) \\
& f & \mapsto & \left[ f\right] =cl\left( f_{\varepsilon }\right) ,%
\end{array}%
\end{equation}%
where $f_{\varepsilon }=f$ , $\forall \varepsilon \in \left] 0,1\right] $.

The following proposition gives the natural embedding of Gevrey
ultradistributions into $\mathcal{G}^{\sigma }\left( \Omega \right) .$

\begin{theorem}
\label{pro-inj}The map
\begin{equation}
\begin{array}{llll}
J_{0}: & E_{3\sigma -1}^{\prime }\left( \Omega \right) & \rightarrow &
\mathcal{G}^{\sigma }\left( \Omega \right) \\
& T & \mapsto & \left[ T\right] =cl\left( \left( T\ast \phi _{\varepsilon
}\right) _{/\Omega }\right) _{\varepsilon }%
\end{array}
\label{mapJ0}
\end{equation}%
is an embedding$.$
\end{theorem}

\begin{proof}
Let $T\in E_{3\sigma -1}^{\prime }\left( \Omega \right) $ with $suppT\subset
K,$ then there exists an $\left( 3\sigma -1\right) $-ultradifferential
operator $P\left( D\right) =\sum\limits_{\gamma \in \mathbb{Z}%
_{+}^{m}}a_{\gamma }D^{\gamma }$ and continuous functions $f_{\gamma }$ with
$suppf_{\gamma }\subset K,\forall \gamma \in \mathbb{Z}_{+}^{m},$ and $%
\sup\limits_{\gamma \in \mathbb{Z}_{+}^{m},x\in K}\left| f_{\gamma }\left(
x\right) \right| \leq M,$ such that
\begin{equation*}
T=\sum_{\gamma \in \mathbb{Z}_{+}^{m}}a_{\gamma }D^{\gamma }f_{\gamma }
\end{equation*}
We have
\begin{equation*}
T\ast \phi _{\varepsilon }\left( x\right) =\sum_{\gamma \in \mathbb{Z}%
_{+}^{m}}a_{\gamma }\frac{\left( -1\right) ^{\left| \gamma \right| }}{%
\varepsilon ^{\left| \gamma \right| }}\int f_{\gamma }\left( x+\varepsilon
y\right) D^{\gamma }\phi \left( y\right) dy
\end{equation*}
Let $\alpha \in \mathbb{Z}_{+}^{m},$ then
\begin{equation*}
\left| \partial ^{\alpha }\left( T\ast \phi _{\varepsilon }\left( x\right)
\right) \right| \leq \sum_{\gamma \in \mathbb{Z}_{+}^{m}}a_{\gamma }\frac{1}{%
\varepsilon ^{\left| \gamma +\alpha \right| }}\int \left| f_{\gamma }\left(
x+\varepsilon y\right) \right| \left| D^{\gamma +\alpha }\phi \left(
y\right) \right| dy
\end{equation*}
From (\ref{12}) and the inequality
\begin{equation}
\left( \beta +\alpha \right) !^{t}\leq 2^{t\left| \beta +\alpha \right|
}\alpha !^{t}\beta !^{t},\text{ }\forall t\geq 1,  \label{6}
\end{equation}
we have, $\forall h>0,\exists c>0,$ such that
\begin{eqnarray*}
\left| \partial ^{\alpha }\left( T\ast \phi _{\varepsilon }\left( x\right)
\right) \right| \leq \sum_{\gamma \in \mathbb{Z}_{+}^{m}}c\frac{h^{\left|
\gamma \right| }}{\gamma !^{3\sigma -1}}\frac{1}{\varepsilon ^{\left| \gamma
+\alpha \right| }}\int \left| f_{\gamma }\left( x+\varepsilon y\right)
\right| \left| D^{\gamma +\alpha }\phi \left( y\right) \right| dy \\
\leq \sum_{\gamma \in \mathbb{Z}_{+}^{m}}c\alpha !^{3\sigma -1}\frac{%
2^{\left( 3\sigma -1\right) \left| \gamma +\alpha \right| }h^{\left| \gamma
\right| }}{\left( \gamma +\alpha \right) !^{2\sigma -1}}\frac{1}{\varepsilon
^{\left| \gamma +\alpha \right| }}b^{\left| \gamma +\alpha \right| }\times \\
\qquad \times \int \left| f_{\gamma }\left( x+\varepsilon y\right) \right|
\frac{\left| D^{\gamma +\alpha }\phi \left( y\right) \right| }{b^{\left|
\gamma +\alpha \right| }\left( \gamma +\alpha \right) !^{\sigma }}dy,
\end{eqnarray*}
then for $h>\frac{1}{2},$
\begin{eqnarray*}
\frac{1}{\alpha !^{3\sigma -1}}\left| \partial ^{\alpha }\left( T\ast \phi
_{\varepsilon }\left( x\right) \right) \right| \leq \left\| \phi \right\|
_{b,\sigma }Mc\sum_{\gamma \in \mathbb{Z}_{+}^{m}}2^{-\left| \gamma \right| }%
\frac{\left( 2^{3\sigma }bh\right) ^{\left| \gamma +\alpha \right| }}{\left(
\gamma +\alpha \right) !^{2\sigma -1}}\frac{1}{\varepsilon ^{\left| \gamma
+\alpha \right| }} \\
\leq c\exp \left( k_{1}\varepsilon ^{-\frac{1}{2\sigma -1}}\right) ,
\end{eqnarray*}
i.e.
\begin{equation}
\left| \partial ^{\alpha }\left( T\ast \phi _{\varepsilon }\left( x\right)
\right) \right| \leq c\left( \alpha \right) \exp \left( k_{1}\varepsilon ^{-%
\frac{1}{2\sigma -1}}\right) ,  \label{1-2}
\end{equation}
where $k_{1}=\left( 2\sigma -1\right) \left( 2^{3\sigma }bh\right) ^{\frac{1%
}{2\sigma -1}}$.

Suppose that $\left( T\ast \phi _{\varepsilon }\right) _{\varepsilon }\in
\mathcal{N}^{\sigma }\left( \Omega \right) ,$ then for every compact $L$ of $%
\Omega ,$ $\exists c>0,$ $\forall k>0,$ $\exists \varepsilon _{0}\in \left]
0,1\right] ,$%
\begin{equation}
\left| T\ast \phi _{\varepsilon }\left( x\right) \right| \leq c\exp \left(
-k\varepsilon ^{-\frac{1}{2\sigma -1}}\right) ,\forall x\in L,\forall
\varepsilon \leq \varepsilon _{0}  \label{7}
\end{equation}
Let $\chi \in D^{3\sigma -1}\left( \Omega \right) $ and $\chi =1$ in a
neighborhood of $K$, then $\forall \psi \in E^{3\sigma -1}\left( \Omega
\right) ,$%
\begin{equation*}
\left\langle T,\psi \right\rangle =\left\langle T,\chi \psi \right\rangle
=\lim_{\varepsilon \rightarrow 0}\int \left( T\ast \phi _{\varepsilon
}\right) \left( x\right) \chi \left( x\right) \psi \left( x\right) dx
\end{equation*}
Consequently, from (\ref{7}), we obtain
\begin{equation*}
\left| \int \left( T\ast \phi _{\varepsilon }\right) \left( x\right) \chi
\left( x\right) \psi \left( x\right) dx\right| \leq c\exp \left(
-k\varepsilon ^{-\frac{1}{2\sigma -1}}\right) ,\forall \varepsilon \leq
\varepsilon _{0},
\end{equation*}
which gives $\left\langle T,\psi \right\rangle =0$
\end{proof}

\begin{remark}
We have $c\left( \alpha \right) =\alpha !^{3\sigma -1}\left\| \phi \right\|
_{b,\sigma }Mc$ in (\ref{1-2}).
\end{remark}

In order to show the commutativity of the following diagram of embeddings
\begin{equation*}
\begin{array}{ccc}
D^{\sigma }\left( \Omega \right) & \rightarrow & \mathcal{G}^{\sigma }\left(
\Omega \right) \\
& \searrow & \uparrow \\
&  & E_{3\sigma -1}^{\prime }\left( \Omega \right)%
\end{array}%
,
\end{equation*}%
we have to prove the following fundamental result.

\begin{proposition}
\label{pro2}Let $f\in D^{\sigma }\left( \Omega \right) $ and $\left( \phi
_{\varepsilon }\right) _{\varepsilon }$ be a net of mollifiers, then
\begin{equation*}
\left( f-\left( f\ast \phi _{\varepsilon }\right) _{/\Omega }\right)
_{\varepsilon }\in \mathcal{N}^{\sigma }\left( \Omega \right)
\end{equation*}
\end{proposition}

\begin{proof}
Let $f\in D^{\sigma }\left( \Omega \right) ,$ then there exists a constant $%
c>0,$ such that
\begin{equation*}
\left\vert \partial ^{\alpha }f\left( x\right) \right\vert \leq
c^{\left\vert \alpha \right\vert +1}\alpha !^{\sigma },\forall \alpha \in
\mathbb{Z}_{+}^{m},\forall x\in \Omega
\end{equation*}%
Let $\alpha \in \mathbb{Z}_{+}^{m}$, the Taylor's formula and the properties
of $\phi _{\varepsilon }$ give
\begin{equation*}
\partial ^{\alpha }\left( f\ast \phi _{\varepsilon }-f\right) \left(
x\right) =\sum_{\left\vert \beta \right\vert =N}\int \frac{\left(
\varepsilon y\right) ^{\beta }}{\beta !}\partial ^{\alpha +\beta }f\left(
\xi \right) \phi \left( y\right) dy,
\end{equation*}%
where $x\leq \xi \leq x+\varepsilon y.$ Consequently, for $b>0$, we have
\begin{eqnarray*}
\left\vert \partial ^{\alpha }\left( f\ast \phi _{\varepsilon }-f\right)
\left( x\right) \right\vert &\leq &\varepsilon ^{N}\sum_{\left\vert \beta
\right\vert =N}\int \frac{\left\vert y\right\vert ^{N}}{\beta !}\left\vert
\partial ^{\alpha +\beta }f\left( \xi \right) \right\vert \left\vert \phi
\left( y\right) \right\vert dy \\
&\leq &\alpha !^{\sigma }\varepsilon ^{N}\sum_{\left\vert \beta \right\vert
=N}\beta !^{2\sigma -1}2^{\sigma \left\vert \alpha +\beta \right\vert
}b^{\left\vert \beta \right\vert }\int \frac{\left\vert \partial ^{\alpha
+\beta }f\left( \xi \right) \right\vert }{\left( \alpha +\beta \right)
!^{\sigma }}\times \\
&&\times \frac{\left\vert y\right\vert ^{\left\vert \beta \right\vert }}{%
b^{\left\vert \beta \right\vert }\beta !^{\sigma }}\left\vert \phi \left(
y\right) \right\vert dy
\end{eqnarray*}%
Let $k>0$ and $T>0,$ then
\begin{eqnarray*}
\left\vert \partial ^{\alpha }\left( f\ast \phi _{\varepsilon }-f\right)
\left( x\right) \right\vert &\leq &\alpha !^{\sigma }\left( \varepsilon
N^{2\sigma -1}\right) ^{N}\left( k^{2\sigma -1}T\right) ^{-N}\times \\
&&\times \sum_{\left\vert \beta \right\vert =N}\int 2^{\sigma \left\vert
\alpha +\beta \right\vert }\left( k^{2\sigma -1}bT\right) ^{\left\vert \beta
\right\vert }\frac{\left\vert \partial ^{\alpha +\beta }f\left( \xi \right)
\right\vert }{\left( \alpha +\beta \right) !^{\sigma }}\times \\
&&\times \frac{\left\vert y\right\vert ^{\left\vert \beta \right\vert }}{%
b^{\left\vert \beta \right\vert }\beta !^{\sigma }}\left\vert \phi \left(
y\right) \right\vert dy \\
&\leq &\alpha !^{\sigma }\left( \varepsilon N^{2\sigma -1}\right) ^{N}\left(
k^{2\sigma -1}T\right) ^{-N}\times \\
&&\times c\left\Vert \phi \right\Vert _{b,\sigma }\left( 2^{\sigma }c\right)
^{\left\vert \alpha \right\vert }\sum_{\left\vert \beta \right\vert
=N}\left( 2^{\sigma }k^{2\sigma -1}bT\right) ^{\left\vert \beta \right\vert
}c^{\left\vert \beta \right\vert },
\end{eqnarray*}%
hence, taking $2^{\sigma }k^{2\sigma -1}bTc\leq \frac{1}{2a},$ with $a>1$,
we obtain
\begin{eqnarray}
\left\vert \partial ^{\alpha }\left( f\ast \phi _{\varepsilon }-f\right)
\left( x\right) \right\vert &\leq &\alpha !^{\sigma }\left( \varepsilon
N^{2\sigma -1}\right) ^{N}\left( k^{2\sigma -1}T\right) ^{-N}\times  \notag
\\
&&\times c\left\Vert \phi \right\Vert _{b,\sigma }\left( 2^{\sigma }c\right)
^{\left\vert \alpha \right\vert }a^{-N}\sum_{\left\vert \beta \right\vert
=N}\left( \frac{1}{2}\right) ^{\left\vert \beta \right\vert }  \notag \\
&\leq &\left\Vert \phi \right\Vert _{b,\sigma }c^{\left\vert \alpha
\right\vert +1}\alpha !^{\sigma }\left( \varepsilon N^{2\sigma -1}\right)
^{N}\left( k^{2\sigma -1}T\right) ^{-N}a^{-N}  \label{1*5}
\end{eqnarray}%
Let $\varepsilon _{0}\in \left] 0,1\right] $ such that $\varepsilon _{0}^{%
\frac{1}{2\sigma -1}}\frac{\ln a}{k}<1$ and take $T>2^{2\sigma -1},$ then
\begin{equation*}
\left( T^{\frac{1}{2\sigma -1}}-1\right) >1>\frac{\ln a}{k}\varepsilon ^{%
\frac{1}{2\sigma -1}},\forall \varepsilon \leq \varepsilon _{0},
\end{equation*}%
in particular, we have
\begin{equation*}
\left( \frac{\ln a}{k}\varepsilon ^{\frac{1}{2\sigma -1}}\right) ^{-1}T^{%
\frac{1}{2\sigma -1}}-\left( \frac{\ln a}{k}\varepsilon ^{\frac{1}{2\sigma -1%
}}\right) ^{-1}>1
\end{equation*}%
Then, there exists $N=N\left( \varepsilon \right) \in \mathbb{Z}^{+},$ such
that
\begin{equation*}
\left( \frac{\ln a}{k}\varepsilon ^{\frac{1}{2\sigma -1}}\right)
^{-1}<N<\left( \frac{\ln a}{k}\varepsilon ^{\frac{1}{2\sigma -1}}\right)
^{-1}T^{\frac{1}{2\sigma -1}},
\end{equation*}%
i.e.
\begin{equation}
1\leq \frac{\ln a}{k}\varepsilon ^{\frac{1}{2\sigma -1}}N\leq T^{\frac{1}{%
2\sigma -1}},  \label{2-2}
\end{equation}%
which gives
\begin{equation*}
a^{-N}\leq \exp \left( -k\varepsilon ^{-\frac{1}{2\sigma -1}}\right) \text{
\ \ \ \ \ and\ \ \ \ \ }\frac{\varepsilon N^{2\sigma -1}}{k^{2\sigma -1}T}%
\leq \left( \frac{1}{\ln a}\right) ^{2\sigma -1}<1,
\end{equation*}%
if we choose $\ln a>1$. Finally, from (\ref{1*5}), we have
\begin{equation}
\left\vert \partial ^{\alpha }\left( f\ast \phi _{\varepsilon }-f\right)
\left( x\right) \right\vert \leq c\exp \left( -k\varepsilon ^{-\frac{1}{%
2\sigma -1}}\right) ,  \label{1*4}
\end{equation}%
i.e. $f\ast \phi _{\varepsilon }-f\in \mathcal{N}^{\sigma }\left( \Omega
\right) $
\end{proof}

From the proof, see (\ref{1*5}), we obtained in fact the following result. \

\begin{corollary}
\label{corol1}Let $f\in D^{\sigma }\left( \Omega \right) $, then for every
compact $K$ of $\Omega ,\forall k>0,\exists c>0,\exists \varepsilon _{0}\in %
\left] 0,1\right] ,\forall \alpha \in \mathbb{Z}_{+}^{m},\forall \varepsilon
\leq \varepsilon _{0}$ ,
\begin{equation}
\sup_{x\in K}\left\vert \partial ^{\alpha }\left( f\ast \phi _{\varepsilon
}-f\right) \left( x\right) \right\vert \leq c^{\left\vert \alpha \right\vert
+1}\alpha !^{\sigma }\exp \left( -k\varepsilon ^{-\frac{1}{2\sigma -1}%
}\right)  \label{1*3}
\end{equation}
\end{corollary}

\section{Sheaf properties of $\mathcal{G}^{\protect\sigma }$}

Let $\Omega ^{\prime }$ be an open subset of $\Omega $ and let $f=\left(
f_{\varepsilon }\right) _{\varepsilon }+\mathcal{N}^{\sigma }\left( \Omega
\right) \in \mathcal{G}^{\sigma }\left( \Omega \right) $, the restriction of
$f$ to $\Omega ^{\prime },$ denoted $f_{/\Omega ^{\prime }}$, is defined as
\begin{equation*}
\left( f_{\varepsilon /\Omega ^{\prime }}\right) _{\varepsilon }+\mathcal{N}%
^{\sigma }\left( \Omega ^{\prime }\right) \in \mathcal{G}^{\sigma }\left(
\Omega ^{\prime }\right)
\end{equation*}

\begin{theorem}
The functor $\Omega \rightarrow \mathcal{G}^{\sigma }\left( \Omega \right) $
is a sheaf of differential algebras on $\mathbb{R}^{n}.$
\end{theorem}

\begin{proof}
Let $\Omega $ be a non void open of $\mathbb{R}^{n}$ and $\left( \Omega
_{\lambda }\right) _{\lambda \in \Lambda }$ be an open covering of $\Omega $%
. we have to show the properties

S1) If $f,g\in \mathcal{G}^{\sigma }\left( \Omega \right) $ such that $%
f_{/\Omega _{\lambda }}=g_{/\Omega _{\lambda }},\forall \lambda \in \Lambda $%
, then $f=g$

S2) If for each $\lambda \in \Lambda ,$ we have $f_{\lambda }\in \mathcal{G}%
^{\sigma }\left( \Omega _{\lambda }\right) $, such that
\begin{equation*}
f_{\lambda /\Omega _{\lambda }\cap \Omega _{\mu }}=f_{\mu /\Omega _{\lambda
}\cap \Omega _{\mu }}\text{ for all }\lambda ,\mu \in \Lambda \text{ with }%
\Omega _{\lambda }\cap \Omega _{\mu }\neq \phi ,
\end{equation*}%
then there exists a unique $f\in \mathcal{G}^{\sigma }\left( \Omega \right) $
with $f_{/\Omega _{\lambda }}=f_{\lambda },\forall \lambda \in \Lambda .$

Let show S1, take $K$ a compact subset of $\Omega $, then there exist
compact sets $K_{1},K_{2},...,K_{m}$ and indices $\lambda _{1},\lambda
_{2},...,\lambda _{m}\in \Lambda $ such that
\begin{equation*}
K\subset \bigcup\limits_{i=1}^{m}K_{i}\text{ and }K_{i}\subset \Omega
_{\lambda _{i}},
\end{equation*}%
where $\left( f_{\varepsilon }-g_{\varepsilon }\right) _{\varepsilon }$
satisfies the $\mathcal{N}^{\sigma }$-estimate on each $K_{i},$ then it
satisfies the $\mathcal{N}^{\sigma }$-estimate on $K$ which means $\left(
f_{\varepsilon }-g_{\varepsilon }\right) _{\varepsilon }\in \mathcal{N}%
^{\sigma }\left( \Omega \right) .$

To show S2, let $\left( \chi _{j}\right) _{j=1}^{\infty }$ be a $C^{\infty }$%
-partition of unity subordinate to the covering $\left( \Omega _{\lambda
}\right) _{\lambda \in \Lambda }$. Set
\begin{equation*}
f:=\left( f_{\varepsilon }\right) _{\varepsilon }+\mathcal{N}^{\sigma
}\left( \Omega \right) ,
\end{equation*}%
where $f_{\varepsilon }=\sum\limits_{j=1}^{\infty }\chi _{j}f_{\lambda
_{j}\varepsilon }$ and $\left( f_{\lambda _{j}\varepsilon }\right)
_{\varepsilon }$ is a representative of $f_{\lambda _{j}}$. Moreover, we set
$f_{\lambda _{j}\varepsilon }=0$ on $\Omega \backslash \Omega _{\lambda
_{j}} $, so that $\chi _{j}f_{\lambda _{j}\varepsilon }$ is $C^{\infty }$ on
all of $\Omega $. First Let $K$ be compact subset of $\Omega $, we have $%
K_{j}=K\cap supp\chi _{j}$ is a compact subset of $\Omega _{\lambda _{j}}$
and $\left( f_{\lambda _{j}\varepsilon }\right) _{\varepsilon }\in \mathcal{E%
}_{m}^{\sigma }\left( \Omega _{\lambda _{j}}\right) $, then $\left( \chi
_{j}f_{\lambda _{j}\varepsilon }\right) $ satisfies $\mathcal{E}_{m}^{\sigma
}$-estimate on each $K_{j}$, we have $\chi _{j}\left( x\right) \equiv 0$ on $%
K$ except for finite number of $j$, i.e. $\exists N>0$, such that
\begin{equation*}
\sum_{j=1}^{\infty }\chi _{j}f_{\lambda _{j}\varepsilon }\left( x\right)
=\sum_{j=1}^{N}\chi _{j}f_{\lambda _{j}\varepsilon }\left( x\right) ,\forall
x\in K
\end{equation*}%
So $\left( \sum \chi _{j}f_{\lambda _{j}\varepsilon }\right) $ satisfies $%
\mathcal{E}_{m}^{\sigma }$-estimate on $K,$ which means $\left(
f_{\varepsilon }\right) _{\varepsilon }\in \mathcal{E}_{m}^{\sigma }\left(
\Omega \right) .$ It remains to show that $f_{/\Omega _{\lambda
}}=f_{\lambda },\forall \lambda \in \Lambda .$ Let $K$ be a compact subset
of $\Omega _{\lambda }$, choose $N>0$ in such a way that $%
\sum\limits_{j=1}^{N}\chi _{j}\left( x\right) \equiv 1$ on a neighborhood $%
\Omega ^{\prime }$ of $K$ with $\overline{\Omega ^{\prime }}$ is compact of $%
\Omega _{\lambda }$. For $x\in K,$%
\begin{equation*}
f_{\varepsilon }\left( x\right) -f_{\lambda \varepsilon }\left( x\right)
=\sum_{j=1}^{N}\chi _{j}\left( x\right) \left( f_{\lambda _{j}\varepsilon
}\left( x\right) -f_{\lambda \varepsilon }\left( x\right) \right)
\end{equation*}%
Since $\left( f_{\lambda _{j}\varepsilon }-f_{\lambda \varepsilon }\right)
\in \mathcal{N}^{\sigma }\left( \Omega _{\lambda _{j}}\cap \Omega _{\lambda
}\right) $ and $K_{j}=K\cap supp\chi _{j}$ is a compact subset of $\Omega
\cap \Omega _{\lambda _{j}}$, then $\left( \sum\limits_{j=1}^{N}\chi
_{j}\left( f_{\lambda _{j}\varepsilon }-f_{\lambda \varepsilon }\right)
\right) $ satisfies the $\mathcal{N}^{\sigma }$-estimate on $K$. The
uniqueness of such $f$ $\in \mathcal{G}^{\sigma }\left( \Omega \right) $
follows from S1.
\end{proof}

Now it is legitimate to introduce the support of $f\in $ $\mathcal{G}%
^{\sigma }\left( \Omega \right) $ as in the classical case.

\begin{definition}
The support of $f\in $ $\mathcal{G}^{\sigma }\left( \Omega \right) ,$
denoted $supp_{g}^{\sigma }f,$ is the complement of the largest open set $U$
such that $f_{/U}=0$.
\end{definition}

As in \cite{GKOS}, we construct the embedding of $D_{3\sigma -1}^{\prime
}\left( \Omega \right) $\ into $\mathcal{G}^{\sigma }\left( \Omega \right) $
using the sheaf properties of $\mathcal{G}^{\sigma }$. First, choose some
covering $\left( \Omega _{\lambda }\right) _{\lambda \in \Lambda }$ of $%
\Omega $ such that each $\overline{\Omega _{\lambda }}$ is a compact subset
of $\Omega $. Let $\left( \psi _{\lambda }\right) _{\lambda \in \Lambda }$
be a family of elements of $D^{\sigma }\left( \Omega \right) \subset
D^{3\sigma -1}\left( \Omega \right) $ with $\psi _{\lambda }\equiv 1$ in
some neighborhood of $\overline{\Omega _{\lambda }}$. For each $\lambda $ we
define
\begin{equation*}
\begin{array}{cccc}
J_{\lambda } & D_{3\sigma -1}^{\prime }\left( \Omega \right) & \rightarrow &
\mathcal{G}^{\sigma }\left( \Omega _{\lambda }\right) \\
& T & \rightarrow & \left[ T\right] _{\lambda }=cl\left( \left( \psi
_{\lambda }T\ast \phi _{\varepsilon }\right) _{/\Omega _{\lambda }}\right)
_{\varepsilon }%
\end{array}%
\end{equation*}%
One can easily show that $\left( \left( \psi _{\lambda }T\ast \phi
_{\varepsilon }\right) _{/\Omega _{\lambda }}\right) _{\varepsilon }\in
\mathcal{E}_{m}^{\sigma }\left( \Omega _{\lambda }\right) $, see the proof
of theorem \ref{pro-inj}, and that the family $\left( J_{\lambda }\left(
T\right) \right) _{\lambda \in \Lambda }$ is coherent, i.e.
\begin{equation*}
J_{\lambda }\left( T\right) _{/\Omega _{\lambda }\cap \Omega _{\mu }}=J_{\mu
}\left( T\right) _{/\Omega _{\lambda }\cap \Omega _{\mu }},\forall \lambda
,\mu \in \Lambda ,
\end{equation*}%
Then if $\left( \chi _{j}\right) _{j=1}^{\infty }$ is a smooth partition of
unity subordinate to $\left( \Omega _{\lambda }\right) _{\lambda \in \Lambda
}$, the precedent theorem allows the embedding
\begin{equation}
\begin{array}{cccc}
J_{\sigma }: & D_{3\sigma -1}^{\prime }\left( \Omega \right) & \rightarrow &
\mathcal{G}^{\sigma }\left( \Omega \right) \\
& T & \rightarrow & \left[ T\right] =cl\left( \sum\limits_{j=1}^{\infty
}\chi _{j}\left( \psi _{\lambda _{j}}T\ast \phi _{\varepsilon }\right)
\right)%
\end{array}
\label{mapJs}
\end{equation}

We can also embed canonically $D_{3\sigma -1}^{\prime }\left( \Omega \right)
$ into $\mathcal{G}^{\sigma }\left( \Omega \right) ,$ see \cite{Del1} for
the case $D^{\prime }\left( \Omega \right) $ and $\mathcal{G}\left( \Omega
\right) $. Indeed, let $\varphi \in D^{\sigma }\left( B\left( 0,2\right)
\right) ,$ $0\leq \varphi \leq 1,\varphi \equiv 1$ on $B\left( 0,1\right) $
and take $\phi \in S^{\left( \sigma \right) }$, define the function $\rho
_{\varepsilon }$\ by
\begin{equation}
\rho _{\varepsilon }\left( x\right) =\left( \frac{1}{\varepsilon }\right)
^{m}\phi \left( \frac{x}{\varepsilon }\right) \varphi \left( x\left\vert \ln
\varepsilon \right\vert \right)  \label{conv}
\end{equation}%
It is easy to prove that, $\exists c>0$, such that $\forall \alpha \in
\mathbb{Z}_{+}^{m},$
\begin{equation*}
\sup_{x\in \mathbb{R}^{m}}\left\vert \partial ^{\alpha }\rho _{\varepsilon
}\left( x\right) \right\vert \leq c^{\left\vert \alpha \right\vert +1}\alpha
!^{\sigma }\varepsilon ^{-m-\left\vert \alpha \right\vert }
\end{equation*}%
Define the injective map
\begin{equation}
\begin{array}{cccc}
J: & D_{3\sigma -1}^{\prime }\left( \Omega \right) & \rightarrow & \mathcal{G%
}^{\sigma }\left( \Omega \right) \\
& T & \rightarrow & \left[ T\right] =cl\left( \left( T\ast \rho
_{\varepsilon }\right) _{\varepsilon }\right)%
\end{array}
\label{mapJ}
\end{equation}

\begin{proposition}
\label{pro-coinc}The map $J$ coincides on $E_{3\sigma -1}^{\prime }\left(
\Omega \right) $ with $J_{0}.$
\end{proposition}

\begin{proof}
We have to show that for $T\in E_{3\sigma -1}^{\prime }\left( \Omega \right)
,$ the net $\left( T\ast \left( \rho _{\varepsilon }-\phi _{\varepsilon
}\right) \right) _{\varepsilon }\in \mathcal{N}^{\sigma }\left( \Omega
\right) $. For $x\left\vert \ln \varepsilon \right\vert <1,$ we have $%
\partial ^{\alpha }\left( \rho _{\varepsilon }-\phi _{\varepsilon }\right)
\left( x\right) =0$, $\forall \alpha \in \mathbb{Z}_{+}^{m}.$ For $%
x\left\vert \ln \varepsilon \right\vert \geq 1,$ we have
\begin{eqnarray*}
\partial ^{\alpha }\left( \rho _{\varepsilon }-\phi _{\varepsilon }\right)
\left( x\right) &=&\varepsilon ^{-m-\left\vert \alpha \right\vert }\partial
^{\alpha }\phi \left( \frac{x}{\varepsilon }\right) \left( \varphi \left(
x\left\vert \ln \varepsilon \right\vert \right) -1\right) + \\
&&+\varepsilon ^{-m}\sum_{\left\vert \beta \right\vert =1}^{\alpha
}\varepsilon ^{-\left\vert \beta \right\vert }\left\vert \ln \varepsilon
\right\vert ^{\left\vert \alpha -\beta \right\vert }\partial ^{\beta }\phi
\left( \frac{x}{\varepsilon }\right) \partial ^{\alpha -\beta }\varphi
\left( x\left\vert \ln \varepsilon \right\vert \right)
\end{eqnarray*}%
Then, $\exists c>0,\forall b>0,\forall \gamma \in \mathbb{Z}_{+}^{m},$
\begin{eqnarray*}
\left\vert \partial ^{\alpha }\left( \rho _{\varepsilon }-\phi _{\varepsilon
}\right) \left( x\right) \right\vert &\leq &b^{\left\vert \gamma \right\vert
+\left\vert \alpha \right\vert +1}\varepsilon ^{-m-\left\vert \alpha
\right\vert }\alpha !^{\sigma }\gamma !^{\sigma }\left( \frac{\varepsilon }{%
\left\vert x\right\vert }\right) ^{\left\vert \gamma \right\vert
}+\varepsilon ^{-m}\sum_{\left\vert \beta \right\vert =1}^{\alpha
}c^{\left\vert \alpha -\beta \right\vert +1}b^{\left\vert \gamma \right\vert
+\left\vert \beta \right\vert +1}\times \\
&&\times \varepsilon ^{-\left\vert \beta \right\vert }\left\vert \ln
\varepsilon \right\vert ^{\left\vert \alpha -\beta \right\vert }\beta
!^{\sigma }\left( \alpha -\beta \right) !^{\sigma }\gamma !^{\sigma }\left(
\frac{\varepsilon }{\left\vert x\right\vert }\right) ^{\left\vert \gamma
\right\vert } \\
&\leq &b^{\left\vert \gamma \right\vert +\left\vert \alpha \right\vert
+1}\varepsilon ^{-m-\left\vert \alpha \right\vert }\alpha !^{\sigma }\gamma
!^{\sigma }\left( \varepsilon \left\vert \ln \varepsilon \right\vert \right)
^{\left\vert \gamma \right\vert }+\varepsilon ^{-m}\sum_{\left\vert \beta
\right\vert =1}^{\alpha }b^{\left\vert \gamma \right\vert +\left\vert \beta
\right\vert +1}c^{\left\vert \alpha -\beta \right\vert +1}\times \\
&&\times \varepsilon ^{-\left\vert \beta \right\vert }\left\vert \ln
\varepsilon \right\vert ^{\left\vert \alpha -\beta \right\vert }\beta
!^{\sigma }\left( \alpha -\beta \right) !^{\sigma }\gamma !^{\sigma }\left(
\varepsilon \left\vert \ln \varepsilon \right\vert \right) ^{\left\vert
\gamma \right\vert }
\end{eqnarray*}%
So, $\exists c>0,\forall b>0,\forall \gamma \in \mathbb{Z}_{+}^{m}$,
\begin{equation}
\left\vert \partial ^{\alpha }\left( \rho _{\varepsilon }-\phi _{\varepsilon
}\right) \left( x\right) \right\vert \leq b^{\left\vert 2\gamma \right\vert
+\left\vert \alpha \right\vert +1}c^{\left\vert 2\gamma +\alpha \right\vert
+1}\varepsilon ^{\left\vert 2\gamma \right\vert -\left\vert \alpha
\right\vert -m}\left\vert \ln \varepsilon \right\vert ^{\left\vert 2\gamma
\right\vert }\alpha !^{\sigma }\left( 2\gamma \right) !^{\sigma }
\label{2-2-1}
\end{equation}%
As $\varepsilon \in \left] 0,1\right] ,$\ we have $\left\vert \ln
\varepsilon \right\vert ^{2}\leq \varepsilon ^{-1},$ then $\forall \gamma
\in \mathbb{Z}_{+}^{m},$%
\begin{equation*}
\left\vert \ln \varepsilon \right\vert ^{\left\vert 2\gamma \right\vert
}\leq \varepsilon ^{-\left\vert \gamma \right\vert },
\end{equation*}%
then (\ref{2-2-1}) gives $\exists c>0,\forall b>0,\forall \gamma \in \mathbb{%
Z}_{+}^{m}$,
\begin{equation*}
\left\vert \partial ^{\alpha }\left( \rho _{\varepsilon }-\phi _{\varepsilon
}\right) \left( x\right) \right\vert \leq b^{\left\vert 2\gamma \right\vert
+\left\vert \alpha \right\vert +1}c^{\left\vert 2\gamma +\alpha \right\vert
+1}\alpha !^{\sigma }\left( 2\gamma \right) !^{\sigma }\varepsilon
^{\left\vert \gamma \right\vert -\left\vert \alpha \right\vert -m}
\end{equation*}%
Since
\begin{equation*}
\left( 2\gamma \right) !^{\sigma }\leq 2^{\sigma \left\vert 2\gamma
\right\vert }\gamma !^{2\sigma },
\end{equation*}%
then for $N=\left\vert \gamma \right\vert ,$ we obtain
\begin{eqnarray*}
\left\vert \partial ^{\alpha }\left( \rho _{\varepsilon }-\phi _{\varepsilon
}\right) \left( x\right) \right\vert &\leq &b^{2N+\left\vert \alpha
\right\vert +1}c^{2N+\left\vert \alpha \right\vert +1}2^{2N\sigma }\alpha
!^{\sigma }N!^{2\sigma }\varepsilon ^{N-\left\vert \alpha \right\vert -m} \\
&\leq &2^{-N}\left( 2^{2\sigma -1}b^{2}c^{2}\right) ^{N}N!^{2\sigma
}\varepsilon ^{N}\times \left( bc\right) ^{\left\vert \alpha \right\vert
+1}\alpha !^{\sigma }\varepsilon ^{-\left\vert \alpha \right\vert -m} \\
&\leq &c^{\prime }h^{\left\vert \alpha \right\vert +1}\alpha !^{\sigma
}\varepsilon ^{-\left\vert \alpha \right\vert -m}\exp \left( -\left(
2^{2\sigma -1}b^{2}c^{2}\right) ^{-\frac{1}{2\sigma }}\varepsilon ^{-\frac{1%
}{2\sigma }}\right) ,
\end{eqnarray*}%
where $c^{\prime }=cb\sum\limits_{n\geq 0}2^{-N}$ and $h=cb,$ then for any $%
k $ take $b>0$ and $\varepsilon $ sufficiently small such that%
\begin{equation*}
\varepsilon ^{-\left\vert \alpha \right\vert -m}\exp \left( -\left(
2^{2\sigma -1}b^{2}c^{2}\right) ^{-\frac{1}{2\sigma }}\varepsilon ^{-\frac{1%
}{2\sigma }}\right) \leq \exp \left( -k\varepsilon ^{-\frac{1}{2\sigma -1}%
}\right)
\end{equation*}%
Then we have $\exists h>0,\forall k>0,\exists \varepsilon _{0}\in \left] 0,1%
\right] $ such that $\forall \varepsilon \leq \varepsilon _{0},$
\begin{equation}
\left\vert \partial ^{\alpha }\left( \rho _{\varepsilon }-\phi _{\varepsilon
}\right) \left( x\right) \right\vert \leq h^{\left\vert \alpha \right\vert
+1}\alpha !^{\sigma }\exp \left( -k\varepsilon ^{-\frac{1}{2\sigma }}\right)
\leq h^{\left\vert \alpha \right\vert +1}\alpha !^{\sigma }\exp \left(
-k\varepsilon ^{-\frac{1}{2\sigma -1}}\right)  \label{2-2-2}
\end{equation}%
As $E_{3\sigma -1}^{\prime }\left( \Omega \right) \subset E_{\sigma
}^{\prime }\left( \Omega \right) ,$ then $\exists L$ a compact subset of $%
\Omega $ such that $\forall h>0,\exists c>0,$ and
\begin{equation*}
\left\vert T\left( y\right) ,\left( \rho _{\varepsilon }-\phi _{\varepsilon
}\right) \left( x-y\right) \right\vert \leq c\sup_{\alpha \in \mathbb{Z}%
_{+}^{m},y\in L}\frac{h^{\left\vert \alpha \right\vert }}{\alpha !^{\sigma }}%
\left\vert \partial _{y}^{\alpha }\left( \rho _{\varepsilon }-\phi
_{\varepsilon }\right) \left( x-y\right) \right\vert
\end{equation*}%
So for $x\in K$, $y\in L$ and by (\ref{2-2-2}), we obtain%
\begin{equation*}
\left\vert \left( T\ast \left( \rho _{\varepsilon }-\phi _{\varepsilon
}\right) \right) \left( x\right) \right\vert \leq c\exp \left( -k\varepsilon
^{-\frac{1}{2\sigma -1}}\right) ,
\end{equation*}%
which proves that $\left( T\ast \left( \rho _{\varepsilon }-\phi
_{\varepsilon }\right) \right) _{\varepsilon }\in \mathcal{N}^{\sigma
}\left( \Omega \right) .$
\end{proof}

The sheaf properties of $\mathcal{G}^{\sigma }$ and the proof of proposition %
\ref{pro-coinc} show that the embedding $J_{\sigma }$ coincides with the
embedding $J.$ Summing up, we have the following commutative diagram%
\begin{equation*}
\begin{array}{ccc}
E^{\sigma }\left( \Omega \right) & \rightarrow & \mathcal{G}^{\sigma }\left(
\Omega \right) \\
\downarrow & \nearrow &  \\
D_{3\sigma -1}^{\prime }\left( \Omega \right) &  &
\end{array}%
\end{equation*}

\begin{definition}
The space of elements of $\mathcal{G}^{\sigma }\left( \Omega \right) $\ with
compact support is denoted $\mathcal{G}_{C}^{\sigma }\left( \Omega \right) .$
\end{definition}

As in the case of Colombeau generalized functions, it is not difficult to
prove the following result.

\begin{proposition}
The space $\mathcal{G}_{C}^{\sigma }\left( \Omega \right) $ is the space of
elements $f$ of $\mathcal{G}^{\sigma }\left( \Omega \right) $ satisfying :
there exist a representative $\left( f_{\varepsilon }\right) _{\varepsilon
\in \left] 0,1\right] }$\ and a compact subset $K$ of $\Omega $ such that $%
suppf_{\varepsilon }\subset K,\forall \varepsilon \in \left] 0,1\right] .$
\newline
\end{proposition}

\section{Equalities in $\mathcal{G}^{\protect\sigma }\left( \Omega \right) $}

In $\mathcal{G}^{\sigma }\left( \Omega \right) $, we have the strong
equality, denoted $=$, between two elements $f=\left[ \left( f_{\varepsilon
}\right) _{\varepsilon }\right] $ and $g=\left[ \left( g_{\varepsilon
}\right) _{\varepsilon }\right] $, which means that
\begin{equation*}
\left( f_{\varepsilon }-g_{\varepsilon }\right) _{\varepsilon }\in \mathcal{N%
}^{\sigma }\left( \Omega \right)
\end{equation*}%
One can easily check that if $K$ is a compact of $\Omega $ and $f=\left[
\left( f_{\varepsilon }\right) _{\varepsilon }\right] \in \mathcal{G}%
^{\sigma }\left( \Omega \right) $, then $\left( \int_{K}f_{\varepsilon
}\left( x\right) dx\right) _{\varepsilon }$defines an element of $\mathcal{C}%
^{\sigma }$.

We define the equality in the sense of ultradistributions, denoted $\overset{%
t}{\sim }$, where $t\in \left[ \sigma ,3\sigma -1\right] ,$ by
\begin{equation*}
f\overset{t}{\sim }g\Longleftrightarrow \left( \int \left( f_{\varepsilon
}\left( x\right) -g_{\varepsilon }\left( x\right) \right) \varphi \left(
x\right) dx\right) _{\varepsilon }\in \mathcal{N}_{0}^{t},\forall \varphi
\in D^{t}\left( \Omega \right) ,
\end{equation*}%
and we say that $f$ equals $g$ in the sense of ultradistributions.

We say that $f=\left[ \left( f_{\varepsilon }\right) _{\varepsilon }\right] $
is associated to $g=\left[ \left( g_{\varepsilon }\right) _{\varepsilon }%
\right] $, denoted $f\approx g,$ if%
\begin{equation*}
\lim_{\varepsilon \rightarrow 0}\int \left( f_{\varepsilon }-g_{\varepsilon
}\right) \left( x\right) \psi \left( x\right) dx=0,\forall \psi \in
D^{3\sigma -1}\left( \Omega \right)
\end{equation*}%
In particular, we say that $f=\left[ \left( f_{\varepsilon }\right)
_{\varepsilon }\right] \in \mathcal{G}^{\sigma }\left( \Omega \right) $ is
associated to the Gevrey ultradistribution $T\in E_{3\sigma -1}^{\prime
}\left( \Omega \right) $, denoted $f\approx T,$ if
\begin{equation*}
\lim_{\varepsilon \rightarrow 0}\int f_{\varepsilon }\left( x\right) \psi
\left( x\right) dx=\left\langle T,\psi \right\rangle ,\forall \psi \in
D^{3\sigma -1}\left( \Omega \right)
\end{equation*}

The main relationship between these inequalities is giving by the following
results.

\begin{proposition}
Let $f,g\in \mathcal{G}^{\sigma }\left( \Omega \right) ,$ $T\in E_{3\sigma
-1}^{\prime }\left( \Omega \right) $, and $t\in \left[ \sigma ,3\sigma -1%
\right] ,$ then

1) $f=g\Longrightarrow f\overset{t}{\sim }g\Longrightarrow f\overset{\sigma }%
{\sim }g\Longrightarrow f\approx g$

2) $T\approx 0$ in $\mathcal{G}^{\sigma }\left( \Omega \right)
\Longrightarrow T=0$ in $E_{3\sigma -1}^{\prime }\left( \Omega \right) $
\end{proposition}

\begin{proof}
Easy.
\end{proof}

\section{Regular generalized Gevrey ultradistributions}

To develop a local and a microlocal analysis with respect to a
\textquotedblright good space of regular elements\textquotedblright\ one
needs first to define these regular elements, the notion of singular support
and its microlocalization with respect to the class of regular elements.

\begin{definition}
The space of regular elements, denoted $\mathcal{E}_{m}^{\sigma ,\infty
}\left( \Omega \right) $, is the space of $\left( f_{\varepsilon }\right)
_{\varepsilon }\in \left( C^{\infty }\left( \Omega \right) \right) ^{\left]
0,1\right] }$ satisfying, for every compact $K$ of $\Omega $, $\exists k>0,$
$\exists c>0,\exists \varepsilon _{0}\in \left] 0,1\right] ,\forall \alpha
\in \mathbb{Z}_{+}^{m},\forall \varepsilon \leq \varepsilon _{0}$,
\begin{equation*}
\sup_{x\in K}\left\vert \partial ^{\alpha }f_{\varepsilon }\left( x\right)
\right\vert \leq c^{\left\vert \alpha \right\vert +1}\alpha !^{\sigma }\exp
\left( k\varepsilon ^{-\frac{1}{2\sigma -1}}\right)
\end{equation*}
\end{definition}

\begin{proposition}
1) The space $\mathcal{E}_{m}^{\sigma ,\infty }\left( \Omega \right) $ is an
algebra stable by the action of $\sigma $-ultradifferential operators.

2) The space $\mathcal{N}^{\sigma ,\infty }\left( \Omega \right) :=\mathcal{N%
}^{\sigma }\left( \Omega \right) \cap \mathcal{E}_{m}^{\sigma ,\infty
}\left( \Omega \right) $ is an ideal of $\mathcal{E}_{m}^{\sigma ,\infty
}\left( \Omega \right) .$
\end{proposition}

\begin{proof}
1) Let $\left( f_{\varepsilon }\right) _{\varepsilon },\left( g_{\varepsilon
}\right) _{\varepsilon }\in \mathcal{E}_{m}^{\sigma ,\infty }\left( \Omega
\right) $ and $K$ be a compact of $\Omega $, then $\exists k_{1}>0,\exists
c_{1}>0,\exists \varepsilon _{1}\in \left] 0,1\right] $ such that $\forall
x\in K,\forall \alpha \in \mathbb{Z}_{+}^{m},\forall \varepsilon \leq
\varepsilon _{1},$%
\begin{equation*}
\left\vert \partial ^{\alpha }f_{\varepsilon }\left( x\right) \right\vert
\leq c_{1}^{\left\vert \alpha \right\vert +1}{\alpha !^{\sigma }}\exp \left(
k_{1}\varepsilon ^{-\frac{1}{2\sigma -1}}\right)
\end{equation*}

We have also $\exists k_{2}>0,\exists c_{2}>0,\exists \varepsilon _{2}\in %
\left] 0,1\right] $ such that $\forall x\in K,\forall \alpha \in \mathbb{Z}%
_{+}^{m},\forall \varepsilon \leq \varepsilon _{2},$%
\begin{equation*}
\left\vert \partial ^{\alpha }g_{\varepsilon }\left( x\right) \right\vert
\leq c_{2}^{\left\vert \alpha \right\vert +1}{\alpha !^{\sigma }}\exp \left(
k_{2}\varepsilon ^{-\frac{1}{2\sigma -1}}\right)
\end{equation*}

Let $\alpha \in \mathbb{Z}_{+}^{m}$, then
\begin{equation*}
\frac{1}{\alpha !^{\sigma }}\left| \partial ^{\alpha }\left( f_{\varepsilon
}g_{\varepsilon }\right) \left( x\right) \right| \leq \sum_{\beta
=0}^{\alpha }\binom{\alpha }{\beta }\frac{1}{\left( \alpha -\beta \right)
!^{\sigma }}\left| \partial ^{\alpha -\beta }f_{\varepsilon }\left( x\right)
\right| \frac{1}{\beta !^{\sigma }}\left| \partial ^{\beta }g_{\varepsilon
}\left( x\right) \right|
\end{equation*}
Let $\varepsilon \leq \min \left( \varepsilon _{1},\varepsilon _{2}\right) $
and $k=k_{1}+k_{2},$ then we have $\forall \alpha \in \mathbb{Z}%
_{+}^{m},\forall x\in K,$%
\begin{eqnarray*}
\exp \left( -k\varepsilon ^{-\frac{1}{2\sigma -1}}\right) \frac{1}{\alpha
!^{\sigma }}\left| \partial ^{\alpha }\left( f_{\varepsilon }g_{\varepsilon
}\right) \left( x\right) \right| \leq \sum_{\beta =0}^{\alpha }\binom{\alpha
}{\beta }\frac{\exp \left( -k_{1}\varepsilon ^{-\frac{1}{2\sigma -1}}\right)
}{\left( \alpha -\beta \right) !^{\sigma }}\left| \partial ^{\alpha -\beta
}f_{\varepsilon }\left( x\right) \right| \\
\times \frac{\exp \left( -k_{2}\varepsilon ^{-\frac{1}{2\sigma -1}}\right) }{%
\beta !^{\sigma }}\left| \partial ^{\beta }g_{\varepsilon }\left( x\right)
\right| \\
\leq \sum_{\beta =0}^{\alpha }\binom{\alpha }{\beta }c_{1}^{\left| \alpha
-\beta \right| }c_{2}^{\left| \beta \right| } \\
\leq 2^{^{\left| \alpha \right| }}\left( c_{1}+c_{2}\right) ^{\left| \alpha
\right| },
\end{eqnarray*}
i. e. $\left( f_{\varepsilon }\right) _{\varepsilon }\left( g_{\varepsilon
}\right) _{\varepsilon }\in \mathcal{E}_{m}^{\sigma ,\infty }\left( \Omega
\right) .$

Let now $P\left( D\right) =\sum a_{\gamma }D^{\gamma }$ be an$\ \sigma $%
-ultradifferential operator, then $\forall h>0,\exists b>0,$ such that
\begin{eqnarray*}
\exp \left( -k_{1}\varepsilon ^{-\frac{1}{2\sigma -1}}\right) \frac{1}{%
\alpha !^{\sigma }}\left| \partial ^{\alpha }\left( P\left( D\right)
f_{\varepsilon }\left( x\right) \right) \right| \leq \exp \left(
-k_{1}\varepsilon ^{-\frac{1}{2\sigma -1}}\right) \sum_{\gamma \in \mathbb{Z}%
_{+}^{m}}b\frac{h^{\left| \gamma \right| }}{\gamma !^{\sigma }}\frac{1}{%
\alpha !^{\sigma }}\left| \partial ^{\alpha +\gamma }f_{\varepsilon }\left(
x\right) \right| \\
\leq b\exp \left( -k_{1}\varepsilon ^{-\frac{1}{2\sigma -1}}\right)
\sum_{\gamma \in \mathbb{Z}_{+}^{m}}\frac{2^{\sigma \left| \alpha +\gamma
\right| }h^{\left| \gamma \right| }}{\left( \alpha +\gamma \right) !^{\sigma
}}\left| \partial ^{\alpha +\gamma }f_{\varepsilon }\left( x\right) \right|
\\
\leq b\sum_{\gamma \in \mathbb{Z}_{+}^{m}}2^{\sigma \left| \alpha +\gamma
\right| }h^{\left| \gamma \right| }c_{1}^{\left| \alpha +\gamma \right| },
\end{eqnarray*}
hence, for $2^{\sigma }hc_{1}\leq \frac{1}{2},$ we have
\begin{equation*}
\exp \left( -k\varepsilon ^{-\frac{1}{2\sigma -1}}\right) \frac{1}{\alpha
!^{\sigma }}\left| \partial ^{\alpha }\left( P\left( D\right) f_{\varepsilon
}\left( x\right) \right) \right| \leq c^{\prime }\left( 2^{\sigma
}c_{1}\right) ^{\left| \alpha \right| },
\end{equation*}
which shows that $\left( P\left( D\right) f_{\varepsilon }\right)
_{\varepsilon }\in \mathcal{E}_{m}^{\sigma ,\infty }\left( \Omega \right) .$

2) The fact that $\mathcal{N}^{\sigma ,\infty }\left( \Omega \right) =%
\mathcal{N}^{\sigma }\left( \Omega \right) \cap \mathcal{E}_{m}^{\sigma
,\infty }\left( \Omega \right) \subset \mathcal{E}_{m}^{\sigma }\left(
\Omega \right) $ and $\mathcal{N}^{\sigma }\left( \Omega \right) $ is an
ideal of $\mathcal{E}_{m}^{\sigma }\left( \Omega \right) ,$ then $\mathcal{N}%
^{\sigma ,\infty }\left( \Omega \right) $ is an ideal of $\mathcal{E}%
_{m}^{\sigma ,\infty }\left( \Omega \right) $
\end{proof}

Now, we define the Gevrey regular elements of $\mathcal{G}^{\sigma }\left(
\Omega \right) .$

\begin{definition}
The algebra of regular generalized Gevrey ultradistributions of order $%
\sigma >1$, denoted $\mathcal{G}^{\sigma ,\infty }\left( \Omega \right) ,$
is the quotient algebra
\begin{equation*}
\mathcal{G}^{\sigma ,\infty }\left( \Omega \right) =\frac{\mathcal{E}%
_{m}^{\sigma ,\infty }\left( \Omega \right) }{\mathcal{N}^{\sigma ,\infty
}\left( \Omega \right) }
\end{equation*}
\end{definition}

It is clear that $E^{\sigma }\left( \Omega \right) \hookrightarrow \mathcal{G%
}^{\sigma ,\infty }\left( \Omega \right) $, and it is easy to show that $%
\mathcal{G}^{\sigma ,\infty }$ is a subsheaf of $\mathcal{G}^{\sigma }$.
This motivates the following definition.

\begin{definition}
We define the $\mathcal{G}^{\sigma ,\infty }$-singular support of a
generalized Gevrey ultradistribution $f$ $\in \mathcal{G}^{\sigma }\left(
\Omega \right) ,$ denoted $\sigma $- $singsupp_{g}\left( f\right) , $ as the
complement of the largest open set $\Omega ^{\prime }$ such that $f\in
\mathcal{G}^{\sigma ,\infty }\left( \Omega ^{\prime }\right) .$
\end{definition}

The following result is a Paley-Wiener type characterization of $\mathcal{G}%
^{\sigma ,\infty }\left( \Omega \right) .$

\begin{proposition}
\label{pro1}Let $f=cl\left( f_{\varepsilon }\right) _{\varepsilon }\in
\mathcal{G}_{C}^{\sigma }\left( \Omega \right) ,$ then $f$ is regular if and
only if $\exists k_{1}>0,\exists k_{2}>0,\exists c>0,\exists \varepsilon
_{1}>0,\forall \varepsilon \leq \varepsilon _{1}$ , such that
\begin{equation}
\left\vert \mathcal{F}\left( f_{\varepsilon }\right) \left( \xi \right)
\right\vert \leq c\exp \left( k_{1}\varepsilon ^{-\frac{1}{2\sigma -1}%
}-k_{2}\left\vert \xi \right\vert ^{\frac{1}{\sigma }}\right) ,\forall \xi
\in \mathbb{R}^{m},  \label{3-2}
\end{equation}%
where $\mathcal{F}\left( f_{\varepsilon }\right) $ denote Fourier transform
of $f_{\varepsilon }.$
\end{proposition}

\begin{proof}
Suppose that $f=cl\left( f_{\varepsilon }\right) _{\varepsilon }\in \mathcal{%
G}_{C}^{\sigma }\left( \Omega \right) \cap \mathcal{G}^{\sigma ,\infty
}\left( \Omega \right) ,$ then $\exists k_{1}>0,\exists c_{1}>0,\exists
\varepsilon _{1}>0,$ $\forall \alpha \in \mathbb{Z}_{+}^{m},\forall x\in
K,\forall \varepsilon \leq \varepsilon _{1},$ $suppf_{\varepsilon }\subset
K, $ such that
\begin{equation*}
\left\vert \partial ^{\alpha }f_{\varepsilon }\right\vert \leq
c_{1}^{\left\vert \alpha \right\vert +1}\alpha !^{\sigma }\exp \left(
k_{1}\varepsilon ^{-\frac{1}{2\sigma -1}}\right)
\end{equation*}%
Consequently we have, $\forall \alpha \in \mathbb{Z}_{+}^{m},$
\begin{equation*}
\left\vert \xi ^{\alpha }\right\vert \left\vert \mathcal{F}\left(
f_{\varepsilon }\right) \left( \xi \right) \right\vert \leq \left\vert \int
\exp \left( -ix\xi \right) \partial ^{\alpha }f_{\varepsilon }\left(
x\right) dx\right\vert ,
\end{equation*}%
then, $\exists c>0,\forall \varepsilon \leq \varepsilon _{1},$%
\begin{equation*}
\left\vert \xi \right\vert ^{\left\vert \alpha \right\vert }\left\vert
\mathcal{F}\left( f_{\varepsilon }\right) \left( \xi \right) \right\vert
\leq c^{\left\vert \alpha \right\vert +1}\alpha !^{\sigma }\exp \left(
k_{1}\varepsilon ^{-\frac{1}{2\sigma -1}}\right)
\end{equation*}%
For $\alpha \in \mathbb{Z}_{+}^{m},\exists N\in \mathbb{Z}_{+}$ such that
\begin{equation*}
\frac{N}{\sigma }\leq \left\vert \alpha \right\vert <\frac{N}{\sigma }+1,
\end{equation*}%
so
\begin{eqnarray*}
\left\vert \xi \right\vert ^{\frac{N}{\sigma }}\left\vert \mathcal{F}\left(
f_{\varepsilon }\right) \left( \xi \right) \right\vert &\leq &c^{\left\vert
\alpha \right\vert +1}\left\vert \alpha \right\vert ^{\left\vert \alpha
\right\vert \sigma }\exp \left( k_{1}\varepsilon ^{-\frac{1}{2\sigma -1}%
}\right) \\
&\leq &c^{N+1}N^{N}\exp \left( k_{1}\varepsilon ^{-\frac{1}{2\sigma -1}%
}\right)
\end{eqnarray*}%
Hence $\exists c>0,\forall N\in \mathbb{Z}^{+},$%
\begin{equation*}
\left\vert \mathcal{F}\left( f_{\varepsilon }\right) \left( \xi \right)
\right\vert \leq c^{N+1}\left\vert \xi \right\vert ^{-\frac{N}{\sigma }%
}N!\exp \left( k_{1}\varepsilon ^{-\frac{1}{2\sigma -1}}\right) ,
\end{equation*}%
which gives
\begin{equation*}
\left\vert \mathcal{F}\left( f_{\varepsilon }\right) \left( \xi \right)
\right\vert \exp \left( \frac{1}{2c}\left\vert \xi \right\vert ^{\frac{1}{%
\sigma }}\right) \leq c\exp \left( k_{1}\varepsilon ^{-\frac{1}{2\sigma -1}%
}\right) \sum 2^{-N},
\end{equation*}%
or
\begin{equation*}
\left\vert \mathcal{F}\left( f_{\varepsilon }\right) \left( \xi \right)
\right\vert \leq c^{\prime }\exp \left( k_{1}\varepsilon ^{-\frac{1}{2\sigma
-1}}-\frac{1}{2c}\left\vert \xi \right\vert ^{\frac{1}{\sigma }}\right) ,
\end{equation*}%
i.e. we have (\ref{3-2}).

Suppose now that (\ref{3-2}) is valid, then $\forall \varepsilon \leq
\varepsilon _{0},$
\begin{equation*}
\left| \partial ^{\alpha }f_{\varepsilon }\left( x\right) \right| \leq
c_{1}\exp \left( k_{1}\varepsilon ^{-\frac{1}{2\sigma -1}}\right) \int
\left| \xi ^{\alpha }\right| \exp \left( -k_{2}\left| \xi \right| ^{\frac{1}{%
\sigma }}\right) d\xi
\end{equation*}
Due to the inequality $t^{N}\leq N!\exp \left( t\right) ,\forall t>0,$ then $%
\exists c_{2}=c\left( k_{2}\right) $ such that
\begin{equation*}
\left| \xi ^{\alpha }\right| \exp \left( -\frac{k_{2}}{2}\left| \xi \right|
^{\frac{1}{\sigma }}\right) \leq c_{2}^{\left| \alpha \right| }\alpha
!^{\sigma },
\end{equation*}
then
\begin{eqnarray*}
\left| \partial ^{\alpha }f_{\varepsilon }\left( x\right) \right| \leq
c_{1}\exp \left( k_{1}\varepsilon ^{-\frac{1}{2\sigma -1}}\right)
c_{2}^{\left| \alpha \right| }\alpha !^{\sigma }\int \exp \left( -\frac{k_{2}%
}{2}\left| \xi \right| ^{\frac{1}{\sigma }}\right) d\xi \\
\leq c^{\left| \alpha \right| +1}\alpha !^{\sigma }\exp \left(
k_{1}\varepsilon ^{-\frac{1}{2\sigma -1}}\right) ,\text{ }
\end{eqnarray*}
where $c=\max \left( c_{1}\int \exp \left( -\frac{k_{2}}{2}\left| \xi
\right| ^{\frac{1}{\sigma }}\right) d\xi ,c_{2}\right) $, i.e. $f\in
\mathcal{G}^{\sigma ,\infty }\left( \Omega \right) $
\end{proof}

\begin{remark}
It is easy to see if $f=cl\left( f_{\varepsilon }\right) _{\varepsilon }\in
\mathcal{G}_{C}^{\sigma }\left( \Omega \right) ,$ then $\exists
k_{1}>0,\exists c>0,\exists \varepsilon _{0}>0,\forall k_{2}>0,\forall
\varepsilon \leq \varepsilon _{0},$%
\begin{equation}
\left\vert \mathcal{F}\left( f_{\varepsilon }\right) \left( \xi \right)
\right\vert \leq {c\exp \left( k_{1}\varepsilon ^{-\frac{1}{2\sigma -1}%
}+k_{2}\left\vert \xi \right\vert ^{\frac{1}{\sigma }}\right) ,}\forall \xi
\in \mathbb{R}^{m}  \label{ref3}
\end{equation}
\end{remark}

The algebra $\mathcal{G}^{\sigma ,\infty }\left( \Omega \right) $ plays the
same role as the Oberguggenberger subalgebra of regular elements $\mathcal{G}%
^{\infty }\left( \Omega \right) $ in the Colombeau algebra $\mathcal{G}%
\left( \Omega \right) $, see \cite{Ober}$.$

\begin{theorem}
We have
\begin{equation*}
\mathcal{G}^{\sigma ,\infty }\left( \Omega \right) \cap D_{3\sigma
-1}^{\prime }\left( \Omega \right) =E^{\sigma }\left( \Omega \right)
\end{equation*}
\end{theorem}

\begin{proof}
Let $S\in \mathcal{G}^{\sigma ,\infty }\left( \Omega \right) \cap D_{3\sigma
-1}^{\prime }\left( \Omega \right) $, for any fixed $x_{0}\in \Omega $ we
take $\psi \in D^{3\sigma -1}\left( \Omega \right) $ with $\psi \equiv 1$ on
neighborhood $U$ of $x_{0},$ then $T=\psi S\in E_{3\sigma -1}^{\prime
}\left( \Omega \right) .$ Let $\phi _{\varepsilon }$ be a net of mollifiers
with $\check{\phi}=\phi $ and let $\chi \in D^{\sigma }\left( \Omega \right)
$ such that $\chi \equiv 1$ on $K=supp\psi .$ As $\left[ T\right] \in
\mathcal{G}^{\sigma ,\infty }\left( \Omega \right) ,$ $\exists
k_{1}>0,\exists k_{2}>0,\exists c_{1}>0,\exists \varepsilon _{1}>0,\forall
\varepsilon \leq \varepsilon _{1},$%
\begin{equation*}
\left\vert \mathcal{F}\left( \chi \left( T\ast \phi _{\varepsilon }\right)
\right) \left( \xi \right) \right\vert \leq c_{1}e^{k_{1}\varepsilon ^{-%
\frac{1}{2\sigma -1}}-k_{2}\left\vert \xi \right\vert ^{\frac{1}{\sigma }}},
\end{equation*}%
then
\begin{eqnarray*}
\left\vert \mathcal{F}\left( \chi \left( T\ast \phi _{\varepsilon }\right)
\right) \left( \xi \right) -\mathcal{F}\left( T\right) \left( \xi \right)
\right\vert &=&\left\vert \mathcal{F}\left( \chi \left( T\ast \phi
_{\varepsilon }\right) \right) \left( \xi \right) -\mathcal{F}\left( \chi
T\right) \left( \xi \right) \right\vert \\
&=&\left\vert \left\langle T\left( x\right) ,\left( \chi \left( x\right)
e^{-i\xi x}\right) \ast \phi _{\varepsilon }\left( x\right) -\left( \chi
\left( x\right) e^{-i\xi x}\right) \right\rangle \right\vert
\end{eqnarray*}%
As $E_{3\sigma -1}^{\prime }\left( \Omega \right) \subset E_{\sigma
}^{\prime }\left( \Omega \right) ,$ then $\exists L$ a compact subset of $%
\Omega $ such that $\forall h>0,\exists c>0,$ and
\begin{equation*}
\left\vert \mathcal{F}\left( \chi \left( T\ast \phi _{\varepsilon }\right)
\right) \left( \xi \right) -\mathcal{F}\left( T\right) \left( \xi \right)
\right\vert \leq c\sup_{\alpha \in \mathbb{Z}_{+}^{m},x\in L}\frac{%
h^{\left\vert \alpha \right\vert }}{\alpha !^{\sigma }}\left\vert \left(
\partial _{x}^{\alpha }\left( \chi \left( x\right) e^{-i\xi x}\ast \phi
_{\varepsilon }\left( x\right) -\chi \left( x\right) e^{-i\xi x}\right)
\right) \right\vert
\end{equation*}%
We have $e^{-i\xi }\chi \in D^{\sigma }\left( \Omega \right) ,$ from the
corollary \ref{corol1}, $\forall k_{3}>0,\exists c_{2}>0,\exists \eta
>0,\forall \varepsilon \leq \eta ,$%
\begin{equation*}
\sup_{\alpha \in \mathbb{Z}_{+}^{m},x\in L}\frac{c_{2}^{\left\vert \alpha
\right\vert }}{\alpha !^{\sigma }}\left\vert \partial _{x}^{\alpha }\left(
\chi \left( x\right) e^{-i\xi x}\ast \phi _{\varepsilon }\left( x\right)
-\chi \left( x\right) e^{-i\xi x}\right) \right\vert \leq
c_{2}e^{-k_{3}\varepsilon ^{-\frac{1}{2\sigma -1}}},
\end{equation*}%
so there exists $c^{\prime }=c^{\prime }\left( k_{3}\right) >0$, such that
\begin{equation*}
\left\vert \mathcal{F}\left( T\right) \left( \xi \right) -\mathcal{F}\left(
\chi \left( T\ast \phi _{\varepsilon }\right) \right) \left( \xi \right)
\right\vert \leq c^{\prime }e^{-k_{3}\varepsilon ^{-\frac{1}{2\sigma -1}}}
\end{equation*}%
Let $\varepsilon \leq \min \left( \eta ,\varepsilon _{1}\right) ,$ then
\begin{eqnarray*}
\left\vert \mathcal{F}\left( T\right) \left( \xi \right) \right\vert &\leq
&\left\vert \mathcal{F}\left( T\right) \left( \xi \right) -\mathcal{F}\left(
\chi \left( T\ast \phi _{\varepsilon }\right) \right) \left( \xi \right)
\right\vert +\left\vert \mathcal{F}\left( \chi \left( T\ast \phi
_{\varepsilon }\right) \right) \left( \xi \right) \right\vert \\
&\leq &c^{\prime }e^{-k_{3}\varepsilon ^{-\frac{1}{2\sigma -1}%
}}+c_{1}e^{k_{1}\varepsilon ^{-\frac{1}{2\sigma -1}}-k_{2}\left\vert \xi
\right\vert ^{\frac{1}{\sigma }}}
\end{eqnarray*}%
Take $c=\max \left( c^{\prime },c_{1}\right) ,$ $\varepsilon =\left( \dfrac{%
k_{1}}{\left( k_{2}-r\right) \left\vert \xi \right\vert ^{\frac{1}{\sigma }}}%
\right) ^{2\sigma -1},r\in \left] 0,k_{2}\right[ $ and $k_{3}=\dfrac{k_{1}r}{%
k_{2}-r}$, then $\exists \delta >0,\exists c>0$ such that
\begin{equation*}
\left\vert \mathcal{F}\left( T\right) \left( \xi \right) \right\vert \leq
ce^{-\delta \left\vert \xi \right\vert ^{\frac{1}{\sigma }}},
\end{equation*}%
which means $T=\psi S\in E^{\sigma }\left( \Omega \right) $. As $\psi \equiv
1$ on the neighborhood $U$ of $x_{0}$, then $S\in E^{\sigma }\left( U\right)
$, consequently $S\in E^{\sigma }\left( \Omega \right) ,$ which prove.
\begin{equation*}
\mathcal{G}^{\sigma ,\infty }\left( \Omega \right) \cap D_{3\sigma
-1}^{\prime }\left( \Omega \right) \subset E^{\sigma }\left( \Omega \right)
\end{equation*}%
We have $E^{\sigma }\left( \Omega \right) \subset E^{3\sigma -1}\left(
\Omega \right) \subset D_{3\sigma -1}^{\prime }\left( \Omega \right) $ and $%
E^{\sigma }\left( \Omega \right) \subset \mathcal{G}^{\sigma ,\infty }\left(
\Omega \right) $ then
\begin{equation*}
E^{\sigma }\left( \Omega \right) \subset \mathcal{G}^{\sigma ,\infty }\left(
\Omega \right) \cap D_{3\sigma -1}^{\prime }\left( \Omega \right)
\end{equation*}%
Consequently we have
\begin{equation*}
\mathcal{G}^{\sigma ,\infty }\left( \Omega \right) \cap D_{3\sigma
-1}^{\prime }\left( \Omega \right) =E^{\sigma }\left( \Omega \right)
\end{equation*}
\end{proof}

\section{Generalized Gevrey wave front}

The aim of this section is to introduce the generalized Gevrey wave front of
a generalized Gevrey ultradistribution and to give its main properties.

\begin{definition}
We define $\sum_{g}^{\sigma }\left( f\right) \subset \mathbb{R}%
^{m}\backslash \left\{ 0\right\} ,f\in \mathcal{G}_{C}^{\sigma }\left(
\Omega \right) $, as the complement of the set of points having a conic
neighborhood $\Gamma $ such that $\exists k_{1}>0,\exists k_{2}>0,\exists
c>0,\exists \varepsilon _{0}\in \left] 0,1\right] ,\forall \xi \in \Gamma
,\forall \varepsilon \leq \varepsilon _{0},$
\begin{equation}
\left\vert \mathcal{F}\left( f_{\varepsilon }\right) \left( \xi \right)
\right\vert \leq c\exp \left( k_{1}\varepsilon ^{-\frac{1}{2\sigma -1}%
}-k_{2}\left\vert \xi \right\vert ^{\frac{1}{\sigma }}\right)  \label{3-1}
\end{equation}
\end{definition}

The following essential properties of $\sum_{g}^{\sigma }\left( f\right) $
are sufficient to define later the generalized Gevrey wave front of
generalized Gevrey ultradistribution.

\begin{proposition}
\label{ref5}For every $f\in \mathcal{G}_{C}^{\sigma }\left( \Omega \right) $%
, we have

1. The set $\sum_{g}^{\sigma }\left( f\right) $ is a closed cone.

2. $\sum_{g}^{\sigma }\left( f\right) =\emptyset \Longleftrightarrow f\in
\mathcal{G}^{\sigma ,\infty }\left( \Omega \right) .$

3. $\sum_{g}^{\sigma }\left( \psi f\right) \subset \sum_{g}^{\sigma }\left(
f\right) ,\forall \psi \in E^{\sigma }\left( \Omega \right) .$
\end{proposition}

\begin{proof}
One can easily, from definition and proposition \ref{pro1}, prove the
assertions 1 and 2.

Let suppose that $\xi _{0}\notin \sum_{g}^{\sigma }\left( f\right) $, then $%
\exists \Gamma $ a conic neighborhood of $\xi _{0},\exists k_{1}>0,\exists
k_{2}>0,\exists c_{1}>0,\exists \varepsilon _{1}\in \left] 0,1\right] ,$ $%
\forall \xi \in \Gamma ,\forall \varepsilon \leq \varepsilon _{1},$
\begin{equation}
\left\vert \mathcal{F}\left( f_{\varepsilon }\right) \left( \xi \right)
\right\vert \leq c_{1}\exp \left( k_{1}\varepsilon ^{-\frac{1}{2\sigma -1}%
}-k_{2}\left\vert \xi \right\vert ^{\frac{1}{\sigma }}\right)  \label{3-02}
\end{equation}%
Let $\chi \in D^{\sigma }\left( \Omega \right) ,$ $\chi \equiv 1$ on
neighborhood of $suppf$, so $\chi \psi \in D^{\sigma }\left( \Omega \right) $%
, hence $\exists k_{3}>0,\exists c_{2}>0,\forall \xi \in \mathbb{R}^{m},$
\begin{equation}
\left\vert \mathcal{F}\left( \chi \psi \right) \left( \xi \right)
\right\vert \leq c_{2}\exp \left( -k_{3}\left\vert \xi \right\vert ^{\frac{1%
}{\sigma }}\right) ,  \label{3-3}
\end{equation}%
Let $\Lambda $ be a conic neighborhood of $\xi _{0}$ such that, $\overline{%
\Lambda }\subset \Gamma ,$ we have, for a fixed $\xi \in \Lambda $,
\begin{eqnarray*}
\mathcal{F}\left( \psi f_{\varepsilon }\right) \left( \xi \right) &=&%
\mathcal{F}\left( \chi \psi f_{\varepsilon }\right) \left( \xi \right) \\
&=&\int_{A}\mathcal{F}\left( f_{\varepsilon }\right) \left( \eta \right)
\mathcal{F}\left( \chi \psi \right) \left( \eta -\xi \right) d\eta +\int_{B}%
\mathcal{F}\left( f_{\varepsilon }\right) \left( \eta \right) \mathcal{F}%
\left( \chi \psi \right) \left( \eta -\xi \right) d\eta ,
\end{eqnarray*}%
where $A=\left\{ \eta \in \mathbb{R}^{m};\left\vert \xi -\eta \right\vert ^{%
\frac{1}{\sigma }}\leq \delta \left( \left\vert \xi \right\vert ^{\frac{1}{%
\sigma }}+\left\vert \eta \right\vert ^{\frac{1}{\sigma }}\right) \right\} $%
; $B=\left\{ \eta \in \mathbb{R}^{m};\left\vert \xi -\eta \right\vert ^{%
\frac{1}{\sigma }}>\delta \left( \left\vert \xi \right\vert ^{\frac{1}{%
\sigma }}+\left\vert \eta \right\vert ^{\frac{1}{\sigma }}\right) \right\} $%
. We choose $\delta $ sufficiently small such that $A\subset \Gamma $ and $%
\dfrac{\left\vert \xi \right\vert }{2^{\sigma }}<\left\vert \eta \right\vert
<2^{\sigma }\left\vert \xi \right\vert ,\forall \eta \in A$. Then \ for $%
\varepsilon \leq \varepsilon _{1},$
\begin{eqnarray}
\left\vert \int_{A}\mathcal{F}\left( f_{\varepsilon }\right) \left( \eta
\right) \mathcal{F}\left( \chi \psi \right) \left( \eta -\xi \right) d\eta
\right\vert &\leq &c_{1}c_{2}\exp \left( k_{1}\varepsilon ^{-\frac{1}{%
2\sigma -1}}-\frac{k_{2}}{2}\left\vert \xi \right\vert ^{\frac{1}{\sigma }%
}\right) \times  \notag \\
&&\times \left\vert \int_{A}\exp \left( -k_{3}\left\vert \eta -\xi
\right\vert ^{\frac{1}{\sigma }}\right) d\eta \right\vert  \notag \\
&\leq &c\exp \left( k_{1}\varepsilon ^{-\frac{1}{2\sigma -1}}-\frac{k_{2}}{2}%
\left\vert \xi \right\vert ^{\frac{1}{\sigma }}\right)  \label{3-4}
\end{eqnarray}%
As $f\in \mathcal{G}_{C}^{\sigma }\left( \Omega \right) ,$ from (\ref{ref3}%
), $\exists c_{3}>0,\exists \mu _{1}>0,\exists \varepsilon _{2}\in \left] 0,1%
\right] ,\forall \mu _{2}>0,\forall \xi \in \mathbb{R}^{m},\forall
\varepsilon \leq \varepsilon _{2},$ such that
\begin{equation*}
\left\vert \mathcal{F}\left( f_{\varepsilon }\right) \left( \xi \right)
\right\vert \leq c_{3}\exp \left( \mu _{1}\varepsilon ^{-\frac{1}{2\sigma -1}%
}+\mu _{2}\left\vert \xi \right\vert ^{\frac{1}{\sigma }}\right) ,
\end{equation*}%
hence, for $\varepsilon \leq \min \left( \varepsilon _{1},\varepsilon
_{2}\right) ,$ we have
\begin{eqnarray*}
\left\vert \int_{B}\mathcal{F}\left( f_{\varepsilon }\right) \left( \eta
\right) \mathcal{F}\left( \chi \psi \right) \left( \eta -\xi \right) d\eta
\right\vert &\leq &c_{2}c_{3}\exp \left( \mu _{1}\varepsilon ^{-\frac{1}{%
2\sigma -1}}\right) \int_{B}\exp \left( \mu _{2}\left\vert \eta \right\vert
^{\frac{1}{\sigma }}-k_{3}\left\vert \eta -\xi \right\vert ^{\frac{1}{\sigma
}}\right) d\eta \\
&\leq &c\exp \left( \mu _{1}\varepsilon ^{-\frac{1}{2\sigma -1}}\right)
\int_{B}\exp \left( \mu _{2}\left\vert \eta \right\vert ^{\frac{1}{\sigma }%
}-k_{3}\delta \left( \left\vert \xi \right\vert ^{\frac{1}{\sigma }%
}+\left\vert \eta \right\vert ^{\frac{1}{\sigma }}\right) \right) d\eta ,
\end{eqnarray*}%
then, taking $\mu _{2}<k_{3}\delta ,$ we obtain
\begin{equation}
\left\vert \int_{B}\mathcal{F}\left( f_{\varepsilon }\right) \left( \eta
\right) \mathcal{F}\left( \chi \psi \right) \left( \eta -\xi \right) d\eta
\right\vert \leq c\exp \left( \mu _{1}\varepsilon ^{-\frac{1}{2\sigma -1}%
}-k_{3}\delta \left\vert \xi \right\vert ^{\frac{1}{\sigma }}\right)
\label{3-5}
\end{equation}%
Consequently, (\ref{3-4}) and (\ref{3-5}) give $\xi _{0}\notin
\sum_{g}^{\sigma }\left( \psi f\right) .$
\end{proof}

\begin{definition}
Let $f\in \mathcal{G}^{\sigma }\left( \Omega \right) $ and $x_{0}\in \Omega $%
, the cone of $\sigma $-singular directions of $f$ at $x_{0}$, denoted $%
\sum_{g,x_{0}}^{\sigma }\left( f\right) $, is
\begin{equation}
\sum\nolimits_{g,x_{0}}^{\sigma }\left( f\right) =\bigcap \left\{
\sum\nolimits_{g}^{\sigma }\left( \phi f\right) :\phi \in D^{\sigma }\left(
\Omega \right) \text{ and }\phi \equiv 1\text{ on a neighborhood of }%
x_{0}\right\}  \label{3-6}
\end{equation}
\end{definition}

\begin{lemma}
\label{lem1}Let $f\in \mathcal{G}^{\sigma }\left( \Omega \right) $, then
\begin{equation*}
\sum\nolimits_{g,x_{0}}^{\sigma }\left( f\right) =\emptyset
\Longleftrightarrow x_{0}\notin \sigma \text{-}singsupp_{g}\left( f\right)
\end{equation*}
\end{lemma}

\begin{proof}
Let $x_{0}\notin \sigma $- $singsupp_{g}\left( f\right) ,$ i.e. $\exists
U\subset \Omega $ an open neighborhood of $x_{0}$ such that $f\in \mathcal{G}%
^{\sigma ,\infty }\left( U\right) $, let $\phi \in D^{\sigma }\left(
U\right) $ such that $\phi \equiv 1$ on a neighborhood of $x_{0},$ then $%
\phi f\in \mathcal{G}^{\sigma ,\infty }\left( \Omega \right) .$ Hence, from
the proposition \ref{ref5}, $\sum\nolimits_{g}^{\sigma }\left( \phi f\right)
=\emptyset ,$ i.e. $\sum\nolimits_{g,x_{0}}^{\sigma }\left( f\right)
=\emptyset .$

Suppose now $\sum\nolimits_{g,x_{0}}^{\sigma }\left( f\right) =\emptyset ,$
let $r>0$ such that $B\left( x_{0},2r\right) \subset \Omega $ and let $\psi
\in D^{\sigma }\left( B\left( x_{0},2r\right) \right) $ such that $0\leq $ $%
\psi \leq 1$ and $\psi \equiv 1$ on $B\left( x_{0},r\right) $. Let $\psi
_{j}\left( x\right) =\psi \left( 3^{j}(x-x_{0})+x_{0}\right) $ then it is
clear that $supp\left( \psi _{j}\right) \subset B\left( x_{0},\frac{2r}{3^{j}%
}\right) \subset \Omega $ and $\psi _{j}\equiv 1$ on $B\left( x_{0},\frac{r}{%
3^{j}}\right) ,$ we have $\forall \phi \in D^{\sigma }\left( \Omega \right) $
with $\phi \equiv 1$ on a neighborhood $U$ of $x_{0},$ $\exists j\in \mathbb{%
Z}^{+}$ such that $supp\left( \psi _{j}\right) \subset U$, then $\psi
_{j}f_{\varepsilon }=\psi _{j}\phi f_{\varepsilon }$ and from proposition %
\ref{ref5}, we have
\begin{equation*}
\sum\nolimits_{g}^{\sigma }\left( \psi _{j}f\right) \subset
\sum\nolimits_{g}^{\sigma }\left( \phi f\right) ,
\end{equation*}
which gives
\begin{equation}
\bigcap\limits_{j\in \mathbb{Z}^{+}}\left( \sum\nolimits_{g}^{\sigma }\left(
\psi _{j}f\right) \right) =\emptyset  \label{3-7}
\end{equation}
We have $\psi _{j}\equiv 1$ on $supp\left( \psi _{j+1}\right) ,$ then $%
\sum\nolimits_{g}^{\sigma }\left( \psi _{j+1}f\right) \subset
\sum\nolimits_{g}^{\sigma }\left( \psi _{j}f\right) $, so from (\ref{3-7}),
there exists $n\in \mathbb{Z}^{+}$ sufficiently large such that $\left( \psi
_{n}f\right) \in \mathcal{G}^{\sigma ,\infty }\left( \Omega \right) ,$ then $%
f\in \mathcal{G}^{\sigma ,\infty }\left( B\left( x_{0},\frac{r}{3^{n}}%
\right) \right) ,$ which means. $x_{0}\notin \sigma $- $singsupp_{g}\left(
f\right) .$
\end{proof}

Now, we are ready to give the definition of the generalized Gevrey wave
front.

\begin{definition}
A point $\left( x_{0},\xi _{0}\right) \notin WF_{g}^{\sigma }\left( f\right)
\subset \Omega \times \mathbb{R}^{m}\backslash \left\{ 0\right\} $ if $\xi
_{0}\notin \sum\nolimits_{g,x_{0}}^{\sigma }\left( f\right) ,$ i.e. there
exists $\phi \in D^{\sigma }\left( \Omega \right) ,\phi \left( x\right) =1$
neighborhood of $x_{0}$, and conic neighborhood $\Gamma $ of $\xi _{0}$, $%
\exists k_{1}>0,\exists k_{2}>0,\exists c>0,\exists \varepsilon _{0}\in %
\left] 0,1\right] ,$ such that $\forall \xi \in \Gamma ,\forall \varepsilon
\leq \varepsilon _{0},$
\begin{equation*}
\left| \mathcal{F}\left( \phi f_{\varepsilon }\right) \left( \xi \right)
\right| \leq c\exp \left( k_{1}\varepsilon ^{-\frac{1}{2\sigma -1}%
}-k_{2}\left| \xi \right| ^{\frac{1}{\sigma }}\right)
\end{equation*}
\end{definition}

The main properties of the generalized Gevrey wave front $WF_{g}^{\sigma }$
are resumed in the following proposition.

\begin{proposition}
Let $f\in \mathcal{G}^{\sigma }\left( \Omega \right) $, then

1) The projection of $WF_{g}^{\sigma }\left( f\right) $ on $\Omega $ is the $%
\sigma -$ $singsupp_{g}\left( f\right) .$

2) If $f\in \mathcal{G}_{C}^{\sigma }\left( \Omega \right) ,$ then the
projection of $WF_{g}^{\sigma }\left( f\right) $ on $\mathbb{R}%
^{m}\backslash \left\{ 0\right\} $ is $\sum_{g}^{\sigma }\left( f\right) .$

3) $\forall \alpha \in \mathbb{Z}_{+}^{m},WF_{g}^{\sigma }\left( \partial
^{\alpha }f\right) \subset WF_{g}^{\sigma }\left( f\right) .$

4) $\forall g\in \mathcal{G}^{\sigma ,\infty }\left( \Omega \right)
,WF_{g}^{\sigma }\left( gf\right) \subset WF_{g}^{\sigma }\left( f\right) .$
\end{proposition}

\begin{proof}
1) and 2) hold from the definition, the proposition \ref{ref5} and lemma \ref%
{lem1}. 3) Let $\left( x_{0},\xi _{0}\right) \notin WF_{g}^{\sigma }\left(
f\right) $, then $\exists \phi \in D^{\sigma }\left( \Omega \right) ,\phi
\equiv 1$ on a neighborhood $\overline{U}$ of $x_{0}$, there exist a conic
neighborhood $\Gamma $ of $\xi _{0},\exists k_{1}>0,\exists k_{2}>0,\exists
c_{1}>0,\exists \varepsilon _{0}\in \left] 0,1\right] ,$ such that $\forall
\xi \in \Gamma ,\forall \varepsilon \leq \varepsilon _{0},$
\begin{equation}
\left\vert \mathcal{F}\left( \phi f_{\varepsilon }\right) \left( \xi \right)
\right\vert \leq c_{1}\exp \left( k_{1}\varepsilon ^{-\frac{1}{2\sigma -1}%
}-k_{2}\left\vert \xi \right\vert ^{\frac{1}{\sigma }}\right)  \label{3-8}
\end{equation}%
We have, for $\psi \in D^{\sigma }\left( U\right) $ such that $\psi \left(
x_{0}\right) =1$,
\begin{eqnarray*}
\left\vert \mathcal{F}\left( \psi \partial f_{\varepsilon }\right) \left(
\xi \right) \right\vert &=&\left\vert \mathcal{F}\left( \partial \left( \psi
f_{\varepsilon }\right) \right) \left( \xi \right) -\mathcal{F}\left( \left(
\partial \psi \right) f_{\varepsilon }\right) \left( \xi \right) \right\vert
\\
&\leq &\left\vert \xi \right\vert \left\vert \mathcal{F}\left( \psi \phi
f_{\varepsilon }\right) \left( \xi \right) \right\vert +\left\vert \mathcal{F%
}\left( \left( \partial \psi \right) \phi f_{\varepsilon }\right) \left( \xi
\right) \right\vert
\end{eqnarray*}%
As $WF_{g}^{\sigma }\left( \psi f\right) \subset WF_{g}^{\sigma }\left(
f\right) ,$ then (\ref{3-8}) holds for both $\left\vert \mathcal{F}\left(
\psi \phi f_{\varepsilon }\right) \left( \xi \right) \right\vert \ $and $%
\left\vert \mathcal{F}\left( \left( \partial \psi \right) \phi
f_{\varepsilon }\right) \left( \xi \right) \right\vert .$ So
\begin{eqnarray*}
\left\vert \xi \right\vert \left\vert \mathcal{F}\left( \psi \phi
f_{\varepsilon }\right) \left( \xi \right) \right\vert &\leq &c\left\vert
\xi \right\vert \exp \left( k_{1}\varepsilon ^{-\frac{1}{2\sigma -1}%
}-k_{2}\left\vert \xi \right\vert ^{\frac{1}{\sigma }}\right) \\
&\leq &c^{\prime }\exp \left( k_{1}\varepsilon ^{-\frac{1}{2\sigma -1}%
}-k_{3}\left\vert \xi \right\vert ^{\frac{1}{\sigma }}\right) ,
\end{eqnarray*}%
with $c^{\prime }>0,k_{3}>0$ such that $\left\vert \xi \right\vert \leq
c^{\prime }\exp \left( k_{2}-k_{3}\right) \left\vert \xi \right\vert ^{\frac{%
1}{\sigma }}.$ Hence (\ref{3-8}) holds for $\left\vert \mathcal{F}\left(
\psi \partial f_{\varepsilon }\right) \left( \xi \right) \right\vert ,$
which proves $\left( x_{0},\xi _{0}\right) \notin WF_{g}^{\sigma }\left(
\partial f\right) $

4) Let $\left( x_{0},\xi _{0}\right) \notin WF_{g}^{\sigma }\left( f\right) $%
, then $\exists \phi \in D^{\sigma }\left( \Omega \right) ,\phi \equiv 1$ on
a neighborhood $U$ of $x_{0}$, there exist a conic neighborhood $\Gamma $ of
$\xi _{0},$ $\exists k_{1}>0,\exists k_{2}>0,\exists c_{1}>0,\exists
\varepsilon _{0}\in \left] 0,1\right] ,$ such that $\forall \xi \in \Gamma
,\forall \varepsilon \leq \varepsilon _{0},$
\begin{equation*}
\left\vert \mathcal{F}\left( \phi f_{\varepsilon }\right) \left( \xi \right)
\right\vert \leq c_{1}\exp \left( k_{1}\varepsilon ^{-\frac{1}{2\sigma -1}%
}-k_{2}\left\vert \xi \right\vert ^{\frac{1}{\sigma }}\right)
\end{equation*}%
Let $\psi \in D^{\sigma }\left( \Omega \right) $ and $\psi \equiv 1$ on $%
supp\phi ,$ then $\mathcal{F}\left( \phi g_{\varepsilon }f_{\varepsilon
}\right) =\mathcal{F}\left( \psi g_{\varepsilon }\right) \ast \mathcal{F}%
\left( \phi f_{\varepsilon }\right) .$ We have $\psi g\in \mathcal{G}%
^{\sigma ,\infty }\left( \Omega \right) \cap \mathcal{G}_{C}^{\sigma }\left(
\Omega \right) ,$ then $\exists c_{2}>0,$ $\exists k_{3}>0,\exists
k_{4}>0,\exists \varepsilon _{1}>0,\forall \xi \in \mathbb{R}^{m},\forall
\varepsilon \leq \varepsilon _{1},$
\begin{equation*}
\left\vert \mathcal{F}\left( \psi g_{\varepsilon }\right) \left( \xi \right)
\right\vert \leq c_{2}\exp \left( k_{3}\varepsilon ^{-\frac{1}{2\sigma -1}%
}-k_{4}\left\vert \xi \right\vert ^{\frac{1}{\sigma }}\right) ,
\end{equation*}%
so
\begin{equation*}
\mathcal{F}\left( \phi g_{\varepsilon }f_{\varepsilon }\right) \left( \xi
\right) =\int_{A}\mathcal{F}\left( \phi f_{\varepsilon }\right) \left( \eta
\right) \mathcal{F}\left( \psi g_{\varepsilon }\right) \left( \eta -\xi
\right) d\eta +\int_{B}\mathcal{F}\left( \phi f_{\varepsilon }\right) \left(
\eta \right) \mathcal{F}\left( \psi g_{\varepsilon }\right) \left( \eta -\xi
\right) d\eta ,
\end{equation*}%
where $A$ and $B$ are the same as in the proof of proposition \ref{ref5}. By
(\ref{ref3}), we have $\exists c>0,\exists \mu _{1}>0,\forall \mu
_{2}>0,\exists \varepsilon _{2}>0,\forall \xi \in \mathbb{R}^{m},\forall
\varepsilon \leq \varepsilon _{2},$%
\begin{equation*}
\left\vert \mathcal{F}\left( \phi f_{\varepsilon }\right) \left( \xi \right)
\right\vert \leq c\exp \left( \mu _{1}\varepsilon ^{-\frac{1}{2\sigma -1}%
}+\mu _{2}\left\vert \xi \right\vert ^{\frac{1}{\sigma }}\right)
\end{equation*}%
The same steps as the proposition \ref{ref5}\ finish the proof.
\end{proof}

\begin{corollary}
Let $P\left( x,D\right) =\sum\limits_{\left\vert \alpha \right\vert \leq
m}a_{\alpha }\left( x\right) D^{\alpha }$ be a partial differential operator
with $\mathcal{G}^{\sigma ,\infty }\left( \Omega \right) $ coefficients,
then
\begin{equation*}
WF_{g}^{\sigma }\left( P\left( x,D\right) f\right) \subset WF_{g}^{\sigma
}\left( f\right) ,\forall f\in \mathcal{G}^{\sigma }\left( \Omega \right)
\end{equation*}
\end{corollary}

\begin{remark}
The reverse inclusion will give a generalized Gevrey microlocal
hypoellipticity of linear partial differential operators with regular Gevrey
generalized coefficients. The case of generalized\ $\mathcal{G}^{\infty }-$%
microlocal hypoellipticity in Colombeau algebra has been studied recently in
\cite{HorObPil}.
\end{remark}

We need the following lemma to show the relationship between $WF_{g}^{\sigma
}\left( T\right) $ and $WF^{\sigma }\left( T\right) ,$ when $T\in D_{3\sigma
-1}^{\prime }\left( \Omega \right) $.

\begin{lemma}
\label{lem4}Let $\varphi \in D^{\sigma }\left( B\left( 0,2\right) \right) ,$
$0\leq \varphi \leq 1$ and $\varphi \equiv 1$ on $B\left( 0,1\right) ,$ and
let $\phi \in S^{\left( \sigma \right) }$, then $\exists c>0,\exists \nu >0,$
$\exists \varepsilon _{0}>0,\forall \varepsilon \in \left] 0,\varepsilon _{0}%
\right] ,\forall \xi \in \mathbb{R}^{m},$
\begin{equation*}
\left\vert \widehat{\rho _{\varepsilon }}\left( \xi \right) \right\vert \leq
c\varepsilon ^{-m}e^{-\nu \varepsilon ^{\frac{1}{\sigma }}\left\vert \xi
\right\vert ^{\frac{1}{\sigma }}},
\end{equation*}%
where $\rho _{\varepsilon }\left( x\right) =\left( \frac{1}{\varepsilon }%
\right) ^{m}\phi \left( \frac{x}{\varepsilon }\right) \varphi \left(
x\left\vert \ln \varepsilon \right\vert \right) $, and $\widehat{\rho }$
denotes the Fourier transform of $\rho $.
\end{lemma}

\begin{proof}
We have, for $\varepsilon $ sufficiently small,%
\begin{equation*}
\varepsilon ^{m}\leq \left\vert \ln \varepsilon \right\vert ^{-m}\leq 1
\end{equation*}%
Let $\xi \in \mathbb{R}^{m}$, then%
\begin{eqnarray*}
\widehat{\rho _{\varepsilon }}\left( \xi \right) &=&\left\vert \ln
\varepsilon \right\vert ^{-m}\left[ \int_{A}\widehat{\phi }\left(
\varepsilon \left( \xi -\eta \right) \right) \widehat{\varphi }\left( \frac{%
\eta }{\left\vert \ln \varepsilon \right\vert }\right) d\eta +\right. \\
&&\text{ \ \ \ \ \ \ \ \ \ \ \ \ \ \ \ \ \ \ \ \ \ \ \ }\left. \int_{B}\text{
}\widehat{\phi }\left( \varepsilon \left( \xi -\eta \right) \right) \widehat{%
\varphi }\left( \frac{\eta }{\left\vert \ln \varepsilon \right\vert }\right)
d\eta \text{\ }\right] \text{,}
\end{eqnarray*}%
where $A=\left\{ \eta ;\left\vert \xi -\eta \right\vert ^{\frac{1}{\sigma }%
}\leq \delta ^{\frac{1}{\sigma }}\left( \left\vert \xi \right\vert ^{\frac{1%
}{\sigma }}+\left\vert \eta \right\vert ^{\frac{1}{\sigma }}\right) \right\}
$ and$\ B=\left\{ \eta ;\left\vert \xi -\eta \right\vert ^{\frac{1}{\sigma }%
}>\delta ^{\frac{1}{\sigma }}\left( \left\vert \xi \right\vert ^{\frac{1}{%
\sigma }}+\left\vert \eta \right\vert ^{\frac{1}{\sigma }}\right) \right\} $%
. We choose $\delta $ sufficiently small such that $\dfrac{\left\vert \xi
\right\vert }{2^{\sigma }}<\left\vert \eta \right\vert <2^{\sigma
}\left\vert \xi \right\vert ,\forall \eta \in A_{\delta }$. Since $\varphi
\in D^{\sigma }\left( \Omega \right) ,\phi \in S^{\left( \sigma \right) }$,
then $\exists k_{1},k_{2}>0,\exists c_{1},c_{2}>0,\forall \xi \in \mathbb{R}%
^{m},$
\begin{equation*}
\left\vert \widehat{\varphi }\left( \xi \right) \right\vert \leq c_{1}\exp
\left( -k_{1}\left\vert \xi \right\vert ^{\frac{1}{\sigma }}\right) \text{
and }\left\vert \widehat{\phi }\left( \xi \right) \right\vert \leq c_{2}\exp
\left( -k_{2}\left\vert \xi \right\vert ^{\frac{1}{\sigma }}\right)
\end{equation*}%
So
\begin{eqnarray*}
I_{1} &=&\left\vert \ln \varepsilon \right\vert ^{-m}\left\vert \int_{A}%
\widehat{\phi }\left( \varepsilon \left( \xi -\eta \right) \right) \widehat{%
\varphi }\left( \frac{\eta }{\left\vert \ln \varepsilon \right\vert }\right)
d\eta \right\vert \leq c_{1}c_{2}\exp \left( -k_{2}\frac{\left\vert \ln
\varepsilon \right\vert ^{-\frac{1}{\sigma }}\left\vert \xi \right\vert ^{%
\frac{1}{\sigma }}}{2}\right) \times \\
&&\times \int \exp \left( -k_{1}\varepsilon ^{\frac{1}{\sigma }}\left\vert
\xi -\eta \right\vert ^{\frac{1}{\sigma }}\right) d\eta
\end{eqnarray*}%
Let $z=\varepsilon \left( \eta -\xi \right) ,$ then
\begin{equation*}
I_{1}\leq c\varepsilon ^{-m}\exp \left( -\frac{k_{2}}{2}\left\vert \ln
\varepsilon \right\vert ^{-\frac{1}{\sigma }}\left\vert \xi \right\vert ^{%
\frac{1}{\sigma }}\right) \int \exp \left( -k_{1}\left\vert z\right\vert ^{%
\frac{1}{\sigma }}\right) dz\leq c\varepsilon ^{-m}\exp \left( -v\varepsilon
^{\frac{1}{\sigma }}\left\vert \xi \right\vert ^{\frac{1}{\sigma }}\right)
\end{equation*}%
For $I_{2}$, we have%
\begin{eqnarray*}
I_{2} &=&\left\vert \ln \varepsilon \right\vert ^{-m}\left\vert \int_{B}%
\widehat{\phi }\left( \varepsilon \left( \xi -\eta \right) \right) \widehat{%
\varphi }\left( \frac{\eta }{\left\vert \ln \varepsilon \right\vert }\right)
d\eta \right\vert \\
&\leq &c_{1}c_{2}\int_{B}\exp \left( -k_{1}\varepsilon ^{\frac{1}{\sigma }%
}\left\vert \xi -\eta \right\vert ^{\frac{1}{\sigma }}-k_{2}\left\vert \ln
\varepsilon \right\vert ^{-\frac{1}{\sigma }}\left\vert \eta \right\vert ^{%
\frac{1}{\sigma }}\right) d\eta \\
&\leq &c_{1}c_{2}\exp \left( -k_{1}\delta \varepsilon ^{\frac{1}{\sigma }%
}\left\vert \xi \right\vert ^{\frac{1}{\sigma }}\right) \int_{B}\exp \left(
-k_{1}\delta \varepsilon ^{\frac{1}{\sigma }}\left\vert \eta \right\vert ^{%
\frac{1}{\sigma }}-k_{2}\left\vert \ln \varepsilon \right\vert ^{-\frac{1}{%
\sigma }}\left\vert \eta \right\vert ^{\frac{1}{\sigma }}\right) d\eta \\
&\leq &c_{1}c_{2}\exp \left( -k_{1}\delta \varepsilon ^{\frac{1}{\sigma }%
}\left\vert \xi \right\vert ^{\frac{1}{\sigma }}\right) \int_{B}\exp \left(
-k\varepsilon ^{\frac{1}{\sigma }}\left\vert \eta \right\vert ^{\frac{1}{%
\sigma }}\right) d\eta \\
&\leq &c\varepsilon ^{-m}\exp \left( -v\varepsilon ^{\frac{1}{\sigma }%
}\left\vert \xi \right\vert ^{\frac{1}{\sigma }}\right)
\end{eqnarray*}%
Consequently, $\exists c>0,\exists v>0,\exists \varepsilon _{0}>0,,\forall
\varepsilon \leq \varepsilon _{0}$ such that
\begin{equation*}
\left\vert \widehat{\rho _{\varepsilon }}\left( \xi \right) \right\vert \leq
c\varepsilon ^{-m}\exp \left( -v\varepsilon ^{\frac{1}{\sigma }}\left\vert
\xi \right\vert ^{\frac{1}{\sigma }}\right) ,\forall \xi \in \mathbb{R}^{m}
\end{equation*}
\end{proof}

We have the following important result.

\begin{theorem}
Let $T\in D_{3\sigma -1}^{\prime }\left( \Omega \right) \cap \mathcal{G}%
^{\sigma }\left( \Omega \right) ,$ then $WF_{g}^{\sigma }\left( T\right)
=WF^{\sigma }\left( T\right) .$
\end{theorem}

\begin{proof}
Let $S\in E_{3\sigma -1}^{\prime }\left( \Omega \right) \subset E_{\sigma
}^{\prime }\left( \Omega \right) $ and $\psi \in D^{\sigma }\left( \Omega
\right) ,$ we have
\begin{equation*}
\left\vert \mathcal{F}\left( \psi \left( S\ast \phi _{\varepsilon }\right)
\right) \left( \xi \right) -\mathcal{F}\left( \psi S\right) \left( \xi
\right) \right\vert =\left\vert \left\langle S\left( x\right) ,\left( \psi
\left( x\right) e^{-i\xi x}\right) \ast \breve{\phi}_{\varepsilon }\left(
x\right) -\left( \psi \left( x\right) e^{-i\xi x}\right) \right\rangle
\right\vert ,
\end{equation*}%
then $\exists L$ a compact of $\Omega $ such that $\forall h>0,\exists c>0,$
\begin{equation*}
\left\vert \mathcal{F}\left( \psi \left( S\ast \phi _{\varepsilon }\right)
\right) \left( \xi \right) -\mathcal{F}\left( \psi S\right) \left( \xi
\right) \right\vert \leq c\sup_{\alpha \in \mathbb{Z}_{+}^{m},x\in L}\frac{%
h^{\left\vert \alpha \right\vert }}{\alpha !^{\sigma }}\left\vert \left(
\partial _{x}^{\alpha }\left( \psi \left( x\right) e^{-i\xi x}\ast \breve{%
\phi}_{\varepsilon }\left( x\right) -\psi \left( x\right) e^{-i\xi x}\right)
\right) \right\vert
\end{equation*}%
We have $e^{-i\xi }\psi \in D^{\sigma }\left( \Omega \right) ,$ from
corollary \ref{corol1}, $\exists c_{2}>0,\forall k_{0}>0,\exists \eta
>0,\forall \varepsilon \leq \eta ,$%
\begin{equation}
\sup_{\alpha \in \mathbb{Z}_{+}^{m},x\in L}\frac{c_{2}^{\left\vert \alpha
\right\vert }}{\alpha !^{\sigma }}\left\vert \partial _{x}^{\alpha }\left(
\psi \left( x\right) e^{-i\xi x}\ast \breve{\phi}_{\varepsilon }\left(
x\right) -\psi \left( x\right) e^{-i\xi x}\right) \right\vert \leq
c_{2}e^{-k_{0}\varepsilon ^{-\frac{1}{2\sigma -1}}},  \label{wf0}
\end{equation}%
so there exist $c^{\prime }>0,\forall k_{0}>0,\exists \eta >0,\forall
\varepsilon \leq \eta ,$ such that
\begin{equation}
\left\vert \mathcal{F}\left( \psi S\right) \left( \xi \right) -\mathcal{F}%
\left( \psi \left( S\ast \phi _{\varepsilon }\right) \right) \left( \xi
\right) \right\vert \leq c^{\prime }e^{-k_{0}\varepsilon ^{-\frac{1}{2\sigma
-1}}}  \label{wf1}
\end{equation}%
Let $T\in D_{3\sigma -1}^{\prime }\left( \Omega \right) \cap \mathcal{G}%
^{\sigma }\left( \Omega \right) $ and $\left( x_{0},\xi _{0}\right) \notin
WF_{g}^{\sigma }\left( T\right) $, then there exist $\chi \in D^{\sigma
}\left( \Omega \right) ,\chi \left( x\right) =1$ in a neighborhood of $%
x_{0}, $ and a conic neighborhood $\Gamma $ of $\xi _{0}$, $\exists
k_{1}>0,\exists k_{2}>0,\exists c_{1}>0,\exists \varepsilon _{0}\in \left]
0,1\right[ ,$ such that $\forall \xi \in \Gamma ,\forall \varepsilon \leq
\varepsilon _{0}, $%
\begin{equation}
\left\vert \mathcal{F}\left( \chi \left( T\ast \rho _{\varepsilon }\right)
\right) \left( \xi \right) \right\vert \leq c_{1}e^{k_{1}\varepsilon ^{-%
\frac{1}{2\sigma -1}}-k_{2}\left\vert \xi \right\vert ^{\frac{1}{\sigma }}}
\label{wf2}
\end{equation}%
let $\psi \in D^{\sigma }\left( \Omega \right) $ equals $1$ in neighborhood
of $x_{0}$ such that for sufficiently small $\varepsilon $ we have $\chi
\equiv 1$ on $supp\psi +B\left( 0,\frac{2}{\left\vert \ln \varepsilon
\right\vert }\right) $, and let $\varphi \in D^{\sigma }\left( B\left(
0,2\right) \right) ,$ $0\leq \varphi \leq 1$ and $\varphi \equiv 1$ on $%
B\left( 0,1\right) ,$ then there exist $\varepsilon _{0}<1,$ such that $%
\forall \varepsilon <\varepsilon _{0},$
\begin{equation}
\psi \left( T\ast \rho _{\varepsilon }\right) \left( x\right) =\psi \left(
\chi T\ast \rho _{\varepsilon }\right) \left( x\right)  \notag
\end{equation}%
where $\rho _{\varepsilon }\left( x\right) =\dfrac{1}{\varepsilon ^{m}}%
\varphi \left( x\left\vert \ln \varepsilon \right\vert \right) \phi \left(
\dfrac{x}{\varepsilon }\right) $. As $\chi T\in E_{3\sigma -1}^{\prime
}\left( \Omega \right) $, then, from proposition \ref{pro-coinc},
\begin{equation*}
\psi \left( T\ast \rho _{\varepsilon }\right) \left( x\right) =\psi \left(
\chi T\ast \rho _{\varepsilon }\right) \left( x\right) =\psi \left( \chi
T\ast \phi _{\varepsilon }\right) \left( x\right)
\end{equation*}

Let $\varepsilon \leq \min \left( \eta ,\varepsilon _{0}\right) $ and $\xi
\in \Gamma $, we have%
\begin{eqnarray*}
\left\vert \mathcal{F}\left( \psi T\right) \left( \xi \right) \right\vert
&\leq &\left\vert \mathcal{F}\left( \psi T\right) \left( \xi \right) -%
\mathcal{F}\left( \psi \left( T\ast \rho _{\varepsilon }\right) \right)
\left( \xi \right) \right\vert +\left\vert \mathcal{F}\left( \chi \left(
T\ast \rho _{\varepsilon }\right) \right) \left( \xi \right) \right\vert \\
&\leq &\left\vert \mathcal{F}\left( \psi \chi T\right) \left( \xi \right) -%
\mathcal{F}\left( \psi \left( \chi T\ast \phi _{\varepsilon }\right) \right)
\left( \xi \right) \right\vert +\left\vert \mathcal{F}\left( \chi \left(
T\ast \rho _{\varepsilon }\right) \right) \left( \xi \right) \right\vert
\end{eqnarray*}%
then by $\left( \ref{wf1}\right) $and $\left( \ref{wf2}\right) ,$ we obtain
\begin{equation*}
\left\vert \mathcal{F}\left( \psi T\right) \left( \xi \right) \right\vert
\leq c^{\prime }e^{-k_{0}\varepsilon ^{-\frac{1}{2\sigma -1}%
}}+c_{1}e^{k_{1}\varepsilon ^{-\frac{1}{2\sigma -1}}-k_{2}\left\vert \xi
\right\vert ^{\frac{1}{\sigma }}}
\end{equation*}%
Take $c=\max \left( c^{\prime },c_{1}\right) ,$ $\varepsilon =\left( \dfrac{%
k_{1}}{\left( k_{2}-r\right) \left\vert \xi \right\vert ^{\frac{1}{\sigma }}}%
\right) ^{2\sigma -1},r\in \left] 0,k_{2}\right[ ,k_{0}=\dfrac{k_{1}r}{%
k_{2}-r}$, then $\exists \delta >0,\exists c>0,$ such that
\begin{equation*}
\left\vert \mathcal{F}\left( \chi T\right) \left( \xi \right) \right\vert
\leq ce^{-\delta \left\vert \xi \right\vert ^{\frac{1}{\sigma }}},
\end{equation*}%
which proves that $\left( x_{0},\xi _{0}\right) \notin WF^{\sigma }\left(
T\right) ,$ i.e. $WF^{\sigma }\left( T\right) \subset WF_{g}^{\sigma }\left(
T\right) .$

Suppose $\left( x_{0},\xi _{0}\right) \notin WF^{\sigma }\left( T\right) $,
then there exist $\chi \in D^{\sigma }\left( \Omega \right) ,\chi \left(
x\right) =1$ in a neighborhood of $x_{0}$, a conic neighborhood $\Gamma $ of
$\xi _{0}$, $\exists \lambda >0,\exists c_{1}>0,$ such that $\forall \xi \in
\Gamma ,$%
\begin{equation}
\left\vert \mathcal{F}\left( \chi T\right) \left( \xi \right) \right\vert
\leq c_{1}e^{-\lambda \left\vert \xi \right\vert ^{\frac{1}{\sigma }}}
\label{wff1}
\end{equation}%
Let also $\psi \in D^{\sigma }\left( \Omega \right) $ equals $1$ in
neighborhood of $x_{0}$ such that for sufficiently small $\varepsilon $ we
have $\chi \equiv 1$ on $supp\psi +B\left( 0,\frac{2}{\left\vert \ln
\varepsilon \right\vert }\right) $, then there exist $\varepsilon _{0}<1, $
such that $\forall \varepsilon <\varepsilon _{0},$
\begin{equation}
\psi \left( T\ast \rho _{\varepsilon }\right) \left( x\right) =\psi \left(
\chi T\ast \rho _{\varepsilon }\right) \left( x\right)  \notag
\end{equation}

We have%
\begin{equation*}
\mathcal{F}\left( \psi \left( T\ast \rho _{\varepsilon }\right) \right)
\left( \xi \right) =\int \mathcal{F}\left( \psi \right) \left( \xi -\eta
\right) \mathcal{F}\left( \chi T\right) \left( \eta \right) \mathcal{F}%
\left( \rho _{\varepsilon }\right) \left( \eta \right) d\eta
\end{equation*}%
Let $\Lambda $ be a conic neighborhood of $\xi _{0}$ such that, $\overline{%
\Lambda }\subset \Gamma .$ For a fixed $\xi \in \Lambda $, we have
\begin{eqnarray*}
\mathcal{F}\left( \psi \left( \chi T\ast \rho _{\varepsilon }\right) \right)
\left( \xi \right) &=&\int_{A}\mathcal{F}\left( \psi \right) \left( \xi
-\eta \right) \mathcal{F}\left( \chi T\right) \left( \eta \right) \mathcal{F}%
\left( \rho _{\varepsilon }\right) \left( \eta \right) d\eta + \\
&&\int_{B}\mathcal{F}\left( \psi \right) \left( \xi -\eta \right) \mathcal{F}%
\left( \chi T\right) \left( \eta \right) \mathcal{F}\left( \rho
_{\varepsilon }\right) \left( \eta \right) d\eta \text{ \ ,}
\end{eqnarray*}%
where $A=\left\{ \eta ;\left\vert \xi -\eta \right\vert ^{\frac{1}{\sigma }%
}\leq \delta \left( \left\vert \xi \right\vert ^{\frac{1}{\sigma }%
}+\left\vert \eta \right\vert ^{\frac{1}{\sigma }}\right) \right\} $ and$\
B=\left\{ \eta ;\left\vert \xi -\eta \right\vert ^{\frac{1}{\sigma }}>\delta
\left( \left\vert \xi \right\vert ^{\frac{1}{\sigma }}+\left\vert \eta
\right\vert ^{\frac{1}{\sigma }}\right) \right\} $. We choose $\delta $
sufficiently small such that $A\subset \Gamma $ and $\dfrac{\left\vert \xi
\right\vert }{2^{\sigma }}<\left\vert \eta \right\vert <2^{\sigma
}\left\vert \xi \right\vert $. Since $\psi \in D^{\sigma }\left( \Omega
\right) $, then $\exists \mu >0,\exists c_{2}>0,\forall \xi \in \mathbb{R}%
^{m},$
\begin{equation*}
\left\vert \mathcal{F}\left( \psi \right) \left( \xi \right) \right\vert
\leq c_{2}\exp \left( -\mu \left\vert \xi \right\vert ^{\frac{1}{\sigma }%
}\right) ,
\end{equation*}%
Then $\exists c>0,\exists \varepsilon _{0}\in \left] 0,1\right[ ,\forall
\varepsilon \leq \varepsilon _{0},$
\begin{eqnarray}
\left\vert \int_{A}\mathcal{F}\left( \psi \right) \left( \xi -\eta \right)
\mathcal{F}\left( \chi T\right) \left( \eta \right) \mathcal{F}\left( \rho
_{\varepsilon }\right) \left( \eta \right) d\eta \right\vert &\leq &c\exp
\left( -\frac{\lambda }{2}\left\vert \xi \right\vert ^{\frac{1}{\sigma }%
}\right) \times  \notag \\
&&\times \left\vert \int_{A}\exp \left( -\mu \left\vert \eta -\xi
\right\vert ^{\frac{1}{\sigma }}\right) \mathcal{F}\left( \rho _{\varepsilon
}\right) \left( \eta \right) d\eta \right\vert  \notag
\end{eqnarray}%
From lemma \ref{lem4}, $\exists c_{3}>0,\exists \nu >0,$ $\exists
\varepsilon _{0}>0$, such that $\forall \varepsilon \in \left] 0,\varepsilon
_{0}\right] ,$
\begin{equation*}
\left\vert \mathcal{F}\left( \rho _{\varepsilon }\right) \left( \xi \right)
\right\vert \leq c_{3}\varepsilon ^{-m}e^{-\nu \varepsilon ^{\frac{1}{\sigma
}}\left\vert \xi \right\vert ^{\frac{1}{\sigma }}},\forall \xi \in \mathbb{R}%
^{m},
\end{equation*}%
then $\exists c>0,$ such that
\begin{eqnarray*}
\left\vert \int_{A}\mathcal{F}\left( \psi \right) \left( \xi -\eta \right)
\mathcal{F}\left( \chi T\right) \left( \eta \right) \mathcal{F}\left( \rho
_{\varepsilon }\right) \left( \eta \right) d\eta \right\vert &\leq
&c\varepsilon ^{-m}\exp \left( -\frac{\lambda }{2}\left\vert \xi \right\vert
^{\frac{1}{\sigma }}\right) \times \\
&&\times \int_{A}\exp \left( -\mu \left\vert \eta -\xi \right\vert ^{\frac{1%
}{\sigma }}\right) \exp \left( -\nu \varepsilon ^{\frac{1}{\sigma }%
}\left\vert \eta \right\vert ^{\frac{1}{\sigma }}\right) d\eta
\end{eqnarray*}%
We have $\exists k>0,\forall \varepsilon \in \left] 0,\varepsilon _{0}\right]
,$%
\begin{equation}
\varepsilon ^{-m}\exp \left( -\nu \varepsilon ^{\frac{1}{\sigma }}\left\vert
\eta \right\vert ^{\frac{1}{\sigma }}\right) \leq \exp \left( k\varepsilon
^{-\frac{1}{2\sigma -1}}\right) ,  \label{wf22}
\end{equation}%
so
\begin{equation}
\left\vert \int_{A}\mathcal{F}\left( \psi \right) \left( \xi -\eta \right)
\mathcal{F}\left( \chi T\right) \left( \eta \right) \mathcal{F}\left( \rho
_{\varepsilon }\right) \left( \eta \right) d\eta \right\vert \leq c\exp
\left( k\varepsilon ^{-\frac{1}{2\sigma -1}}-\frac{\lambda }{2}\left\vert
\xi \right\vert ^{\frac{1}{\sigma }}\right)  \label{wf3}
\end{equation}%
As $\chi T\in E_{3\sigma -1}^{\prime }\left( \Omega \right) \subset
E_{\sigma }^{\prime }\left( \Omega \right) ,$ then $\forall l>0,\exists
c>0,\forall \xi \in \mathbb{R}^{m},$
\begin{equation*}
\left\vert \mathcal{F}\left( \chi T\right) \left( \xi \right) \right\vert
\leq c\exp \left( l\left\vert \xi \right\vert ^{\frac{1}{\sigma }}\right) ,
\end{equation*}%
hence, we have
\begin{eqnarray*}
\left\vert \int_{B}\mathcal{F}\left( \psi \right) \left( \xi -\eta \right)
\mathcal{F}\left( \chi T\right) \left( \eta \right) \mathcal{F}\left( \rho
_{\varepsilon }\right) \left( \eta \right) d\eta \right\vert &\leq
&c\int_{B}\exp \left( l\left\vert \eta \right\vert ^{\frac{1}{\sigma }}-\mu
\left\vert \eta -\xi \right\vert ^{\frac{1}{\sigma }}\right) \left\vert
\mathcal{F}\left( \rho _{\varepsilon }\right) \right\vert d\eta \\
&\leq &c^{\prime }\varepsilon ^{-m}\exp \left( -\mu \delta \left\vert \xi
\right\vert ^{\frac{1}{\sigma }}\right) \times \\
&&\times \int_{B}\exp \left( \left( l-\mu \delta \right) \left\vert \eta
\right\vert ^{\frac{1}{\sigma }}-\nu \varepsilon ^{\frac{1}{\sigma }%
}\left\vert \eta \right\vert ^{\frac{1}{\sigma }}\right) d\eta ,
\end{eqnarray*}%
then, taking $l-\mu \delta =-a<0$ and using (\ref{wf22})$,$ we obtain for a
constant $c>0,$
\begin{equation}
\left\vert \int_{B}\mathcal{F}\left( \psi \right) \left( \xi -\eta \right)
\mathcal{F}\left( \chi T\right) \left( \eta \right) \mathcal{F}\left( \rho
_{\varepsilon }\right) \left( \eta \right) d\eta \right\vert \leq c\exp
\left( k\varepsilon ^{-\frac{1}{2\sigma -1}}-\mu \delta \left\vert \xi
\right\vert ^{\frac{1}{\sigma }}\right)  \label{wf4}
\end{equation}%
Consequently, (\ref{wf3}) and (\ref{wf4}) give $\exists c>0,\exists
k_{1}>0,\exists k_{2}>0,$%
\begin{equation}
\left\vert \mathcal{F}\left( \psi \left( T\ast \rho _{\varepsilon }\right)
\right) \left( \xi \right) \right\vert \leq c\exp \left( k_{1}\varepsilon ^{-%
\frac{1}{2\sigma -1}}-k_{2}\left\vert \xi \right\vert ^{\frac{1}{\sigma }%
}\right)  \label{wf5}
\end{equation}

which gives that $\left( x_{0},\xi _{0}\right) \notin WF_{g}^{\sigma }\left(
T\right) ,$ so $WF_{g}^{\sigma }\left( T\right) \subset WF^{\sigma }\left(
T\right) $ which ends the proof.
\end{proof}

\section{Generalized H\"{o}rmander's theorem}

To extend the generalized H\"{o}rmander's result on the wave front set of
the product, define $WF_{g}^{\sigma }\left( f\right) +WF_{g}^{\sigma }\left(
g\right) ,$ where $f,g\in \mathcal{G}^{\sigma }\left( \Omega \right) ,$ as
the set
\begin{equation*}
\left\{ \left( x,\xi +\eta \right) ;\left( x,\xi \right) \in WF_{g}^{\sigma
}\left( f\right) ,\left( x,\eta \right) \in WF_{g}^{\sigma }\left( g\right)
\right\}
\end{equation*}%
We recall the following fundamental lemma, see \cite{Hor-Kun} for the proof.

\begin{lemma}
\label{lem2}Let $\sum_{1}$, $\sum_{2}$ be closed cones in $\mathbb{R}%
^{m}\backslash \left\{ 0\right\} ,$ such that $0\notin \sum_{1}+\sum_{2}$ ,
then

i) $\overline{\sum\nolimits_{1}+\sum\nolimits_{2}}^{\mathbb{R}^{m}/\left\{
0\right\} }=\left( \sum\nolimits_{1}+\sum\nolimits_{2}\right) \cup
\sum\nolimits_{1}\cup \sum\nolimits_{2}$

ii) For any open conic neighborhood $\Gamma $ of $\sum\nolimits_{1}+\sum%
\nolimits_{2}$ in $\mathbb{R}^{m}\backslash \left\{ 0\right\} ,$ one can
find open conic neighborhoods of $\Gamma _{1},$ $\Gamma _{2}$ in $\mathbb{R}%
^{m}\backslash \left\{ 0\right\} $ of, respectively, $\sum\nolimits_{1},\sum%
\nolimits_{2}$ , such that
\begin{equation*}
\Gamma _{1}+\Gamma _{2}\subset \Gamma
\end{equation*}
\end{lemma}

The principal result of this section is the following theorem.

\begin{theorem}
Let $f,g\in \mathcal{G}^{\sigma }\left( \Omega \right) $ , such that $%
\forall x\in \Omega ,$%
\begin{equation}
\left( x,0\right) \notin WF_{g}^{\sigma }\left( f\right) +WF_{g}^{\sigma
}\left( g\right) ,  \label{4-1}
\end{equation}
then
\begin{equation}
WF_{g}^{\sigma }\left( fg\right) \subseteq \left( WF_{g}^{\sigma }\left(
f\right) +WF_{g}^{\sigma }\left( g\right) \right) \cup WF_{g}^{\sigma
}\left( f\right) \cup WF_{g}^{\sigma }\left( g\right)  \label{4-2}
\end{equation}
\end{theorem}

\begin{proof}
Let $\left( x_{0},\xi _{0}\right) \notin \left( WF_{g}^{\sigma }\left(
f\right) +WF_{g}^{\sigma }\left( g\right) \right) \cup WF_{g}^{\sigma
}\left( f\right) \cup WF_{g}^{\sigma }\left( g\right) ,$ then $\exists \phi
\in D^{\sigma }\left( \Omega \right) $, $\phi \left( x_{0}\right) =1,$ $\xi
_{0}\notin \left( \sum\nolimits_{g}^{\sigma }\left( \phi f\right)
+\sum\nolimits_{g}^{\sigma }\left( \phi g\right) \right) \cup
\sum\nolimits_{g}^{\sigma }\left( \phi f\right) \cup
\sum\nolimits_{g}^{\sigma }\left( \phi g\right) .$ From (\ref{4-1}) we have $%
0\notin \sum\nolimits_{g}^{\sigma }\left( \phi f\right)
+\sum\nolimits_{g}^{\sigma }\left( \phi g\right) $ then by lemma \ref{lem2}
i), we have
\begin{equation*}
\xi _{0}\notin \left( \sum\nolimits_{g}^{\sigma }\left( \phi f\right)
+\sum\nolimits_{g}^{\sigma }\left( \phi g\right) \right) \cup
\sum\nolimits_{g}^{\sigma }\left( \phi f\right) \cup
\sum\nolimits_{g}^{\sigma }\left( \phi g\right) =\overline{%
\sum\nolimits_{g}^{\sigma }\left( \phi f\right) +\sum\nolimits_{g}^{\sigma
}\left( \phi g\right) }^{\mathbb{R}^{m}\backslash \left\{ 0\right\} }
\end{equation*}
Let $\Gamma _{0}$ be an open conic neighborhood of $\sum\nolimits_{g}^{%
\sigma }\left( \phi f\right) +\sum\nolimits_{g}^{\sigma }\left( \phi
g\right) $ in $\mathbb{R}^{m}\backslash \left\{ 0\right\} $ such that $\xi
_{0}\notin \overline{\Gamma }_{0}$ then, from lemma \ref{lem2} ii), there
exist open cones $\Gamma _{1}$ and $\Gamma _{2}$ in $\mathbb{R}%
^{m}\backslash \left\{ 0\right\} $ such that
\begin{equation*}
\sum\nolimits_{g}^{\sigma }\left( \phi f\right) \subset \Gamma _{1},\text{ }%
\sum\nolimits_{g}^{\sigma }\left( \phi g\right) \subset \Gamma _{2}\text{
and }\Gamma _{1}+\Gamma _{2}\subset \Gamma _{0}
\end{equation*}
Define $\Gamma =\mathbb{R}^{m}\backslash \overline{\Gamma }_{0},$ so
\begin{equation}
\Gamma \cap \Gamma _{2}=\emptyset \text{ and }\left( \Gamma -\Gamma
_{2}\right) \cap \Gamma _{1}=\emptyset  \label{4.3}
\end{equation}
Let $\xi \in \Gamma $ and $\varepsilon \in \left] 0,1\right] $
\begin{eqnarray*}
\mathcal{F}\left( \phi f_{\varepsilon }\phi g_{\varepsilon }\right) \left(
\xi \right) &=&\left( \mathcal{F}\left( \phi f_{\varepsilon }\right) \ast
\mathcal{F}\left( \phi g_{\varepsilon }\right) \right) \left( \xi \right) \\
&=&\int_{\Gamma _{2}}\mathcal{F}\left( \phi f_{\varepsilon }\right) \left(
\xi -\eta \right) \mathcal{F}\left( \phi g_{\varepsilon }\right) \left( \eta
\right) d\eta +\int_{\Gamma _{2}^{c}}\mathcal{F}\left( \phi f_{\varepsilon
}\right) \left( \xi -\eta \right) \mathcal{F}\left( \phi g_{\varepsilon
}\right) \left( \eta \right) d\eta \\
&=&I_{1}\left( \xi \right) +I_{2}\left( \xi \right)
\end{eqnarray*}
From (\ref{4.3}), $\exists c_{1}>0,\exists k_{1},k_{2}>0,\exists \varepsilon
_{1}>0$ such that $\forall \varepsilon \leq \varepsilon _{1},\forall \eta
\in \Gamma _{2},$
\begin{equation*}
\mathcal{F}\left( \phi f_{\varepsilon }\right) \left( \xi -\eta \right) \leq
c_{1}\exp \left( k_{1}\varepsilon ^{-\frac{1}{2\sigma -1}}-k_{2}\left\vert
\xi -\eta \right\vert ^{\frac{1}{\sigma }}\right) ,
\end{equation*}
and from (\ref{ref3}) $\exists c_{2}>0,\exists k_{3}>0,\forall
k_{4}>0,\exists \varepsilon _{2}>0,\forall \eta \in \mathbb{R}^{m},\forall
\varepsilon \leq \varepsilon _{2},$%
\begin{equation*}
\left\vert \mathcal{F}\left( \phi g_{\varepsilon }\right) \left( \eta
\right) \right\vert \leq c_{2}\exp \left( k_{2}\varepsilon ^{-\frac{1}{%
2\sigma -1}}+k_{4}\left\vert \eta \right\vert ^{\frac{1}{\sigma }}\right)
\end{equation*}
Let $\gamma >0$ sufficiently small such that $\left\vert \xi -\eta
\right\vert ^{\frac{1}{\sigma }}\geq \gamma \left( \left\vert \xi
\right\vert ^{\frac{1}{\sigma }}+\left\vert \eta \right\vert ^{\frac{1}{%
\sigma }}\right) ,\forall \eta \in \Gamma _{2}.$ Hence for $\varepsilon \leq
\min \left( \varepsilon _{1},\varepsilon _{2}\right) ,$%
\begin{equation*}
\left\vert I_{1}\left( \xi \right) \right\vert \leq c_{1}c_{2}\exp \left(
\left( k_{1}+k_{2}\right) \varepsilon ^{-\frac{1}{2\sigma -1}}-k_{2}\gamma
\left\vert \xi \right\vert ^{\frac{1}{\sigma }}\right) \int \exp \left(
-k_{2}\gamma \left\vert \eta \right\vert ^{\frac{1}{\sigma }%
}+k_{4}\left\vert \eta \right\vert ^{\frac{1}{\sigma }}\right) d\eta
\end{equation*}
take $k_{4}>k_{2}\gamma ,$ then
\begin{equation}
\left\vert I_{1}\left( \xi \right) \right\vert \leq c^{\prime }\exp \left(
k_{1}^{\prime }\varepsilon ^{-\frac{1}{2\sigma -1}}-k_{2}^{\prime
}\left\vert \xi \right\vert ^{\frac{1}{\sigma }}\right)  \label{4-4}
\end{equation}
Let $r>0,$
\begin{eqnarray*}
I_{2}\left( \xi \right) &=&\int_{\Gamma _{2}^{c}\cap \left\{ \left\vert \eta
\right\vert \leq r\left\vert \xi \right\vert \right\} }\mathcal{F}\left(
\phi f_{\varepsilon }\right) \left( \xi -\eta \right) \mathcal{F}\left( \phi
g_{\varepsilon }\right) \left( \eta \right) d\eta +\int_{\Gamma _{2}^{c}\cap
\left\{ \left\vert \eta \right\vert \geq r\left\vert \xi \right\vert
\right\} }\mathcal{F}\left( \phi f_{\varepsilon }\right) \left( \xi -\eta
\right) \mathcal{F}\left( \phi g_{\varepsilon }\right) \left( \eta \right)
d\eta \\
&=&I_{21}\left( \xi \right) +I_{22}\left( \xi \right)
\end{eqnarray*}
Choose $r$ sufficiently small such that $\left\{ \left\vert \eta \right\vert
^{\frac{1}{\sigma }}\leq r\left\vert \xi \right\vert ^{\frac{1}{\sigma }%
}\right\} \Longrightarrow \xi -\eta \notin \Gamma _{1}.$ Then $\left\vert
\xi -\eta \right\vert ^{\frac{1}{\sigma }}\geq \left( 1-r\right) \left\vert
\xi \right\vert ^{\frac{1}{\sigma }}\geq \left( 1-2r\right) \left\vert \xi
\right\vert ^{\frac{1}{\sigma }}+\left\vert \eta \right\vert ^{\frac{1}{%
\sigma }}$ , consequently $\exists c_{3}>0,\exists \lambda _{1},\lambda
_{2},\lambda _{3}>0,\exists \varepsilon _{3}>0$ such that $\forall
\varepsilon \leq \varepsilon _{3},$
\begin{eqnarray*}
\left\vert I_{21}\left( \xi \right) \right\vert &\leq &c_{3}\exp \left(
\lambda _{1}\varepsilon ^{-\frac{1}{2\sigma -1}}\right) \int \exp \left(
-\lambda _{2}\left\vert \xi -\eta \right\vert ^{\frac{1}{\sigma }}-\lambda
_{3}\left\vert \eta \right\vert ^{\frac{1}{\sigma }}\right) d\eta \\
&\leq &c_{3}\exp \left( \lambda _{1}\varepsilon ^{-\frac{1}{2\sigma -1}%
}-\lambda _{2}^{\prime }\left\vert \xi \right\vert ^{\frac{1}{\sigma }%
}\right) \int \exp \left( -\lambda _{3}^{\prime }\left\vert \eta \right\vert
^{\frac{1}{\sigma }}\right) d\eta \\
&\leq &c_{3}^{\prime }\exp \left( \lambda _{1}\varepsilon ^{-\frac{1}{%
2\sigma -1}}-\lambda _{2}^{\prime }\left\vert \xi \right\vert ^{\frac{1}{%
\sigma }}\right)
\end{eqnarray*}
If $\left\vert \eta \right\vert ^{\frac{1}{\sigma }}\geq r\left\vert \xi
\right\vert ^{\frac{1}{\sigma }},$ we have $\left\vert \eta \right\vert ^{%
\frac{1}{\sigma }}\geq \dfrac{\left\vert \eta \right\vert ^{\frac{1}{\sigma }%
}+r\left\vert \xi \right\vert ^{\frac{1}{\sigma }}}{2}$ , and then $\exists
c_{4}>0,\exists \mu _{1},\mu _{3}>0,\forall \mu _{2}>0,\exists \varepsilon
_{4}>0$ such that $\forall \varepsilon \leq \varepsilon _{4},$
\begin{eqnarray*}
\left\vert I_{21}\left( \xi \right) \right\vert &\leq &c_{4}\exp \left( \mu
_{1}\varepsilon ^{-\frac{1}{2\sigma -1}}\right) \int \exp \left( \mu
_{2}\left\vert \xi -\eta \right\vert ^{\frac{1}{\sigma }}-\mu _{3}\left\vert
\eta \right\vert ^{\frac{1}{\sigma }}\right) d\eta \\
&\leq &c_{4}\exp \left( \mu _{1}\varepsilon ^{-\frac{1}{2\sigma -1}}\right)
\int \exp \left( \mu _{2}\left\vert \xi -\eta \right\vert ^{\frac{1}{\sigma }%
}-\mu _{3}^{\prime }\left\vert \eta \right\vert ^{\frac{1}{\sigma }}-\mu
_{3}^{\prime ^{\prime }}\left\vert \xi \right\vert ^{\frac{1}{\sigma }%
}\right) d\eta ,
\end{eqnarray*}
if take $\mu _{2}<\frac{\mu _{3}^{\prime }}{2}\left( 1+\frac{1}{r}\right) ,$
we obtain
\begin{equation*}
\left\vert I_{21}\left( \xi \right) \right\vert \leq c_{4}^{\prime }\exp
\left( k_{3}^{\prime }\varepsilon ^{-\frac{1}{2\sigma -1}}-\mu _{3}^{\prime
^{\prime }}\left\vert \xi \right\vert ^{\frac{1}{\sigma }}\right) ,
\end{equation*}
which finishes the proof.
\end{proof}

\end{document}